\documentclass[11pt]{article}

\usepackage{latexsym,amsmath,amssymb,amstext,amsthm,euscript,array,enumerate,mathrsfs}
\usepackage{color}
\usepackage{comment}
\usepackage[utf8]{inputenc}   
\usepackage[english]{babel} 
\usepackage{hyperref} 
\hypersetup{pdfstartview=, hidelinks} 
\usepackage[numbers]{natbib} 
\usepackage{bbm} 
\usepackage{microtype} 
\usepackage{url} 
\allowdisplaybreaks 

\theoremstyle{plain}
\newtheorem{Theorem}{Theorem}[section]
\newtheorem{Lemma}[Theorem]{Lemma}
\newtheorem{Proposition}[Theorem]{Proposition}
\newtheorem{Corollary}[Theorem]{Corollary}
\newtheorem{Definition}[Theorem]{Definition}
\newtheorem{Hp}[Theorem]{Assumptions}

\theoremstyle{definition}
\newtheorem{Remark}[Theorem]{Remark}

\numberwithin{equation}{section}

\def\ds{\begin{displaystyle}}
	\def\eds{\end{displaystyle}}

\def\<{\langle }
\def\>{\rangle }

\def \R{\mathbb{R}}

\def \E{\mathbb{E}} 
\def \F{\mathbb{F}} 

\def \P{\mathbb{P}}

\def\cala{{\cal A}}

\def\calf{{\cal F}}
\def\calg{{\cal G}}

\def\calp{{\cal P}}

\def\cals{{\cal S}}

\def\PPtwo{{\calp_2(\R^d)}} 

\def\beqs{\begin{eqnarray*}}
	\def\enqs{\end{eqnarray*}}
\def\beq{\begin{eqnarray}}
	\def\enq{\end{eqnarray}}

\allowdisplaybreaks

\title{Ergodic control of McKean-Vlasov systems \\ on the Wasserstein space}
\author{
	Marco Fuhrman\thanks{Universit\`a degli Studi di Milano, Dipartimento di Matematica, Via Saldini 50, 20133 Milano, Italy;
		\newline
		e-mail: \texttt{marco.fuhrman@unimi.it} The author is a member of INDAM-GNAMPA.},
	Silvia Rudà 
	\thanks{Universit\`a degli Studi di Milano, Dipartimento di Matematica, Via Saldini 50, 20133 Milano, Italy;
		\newline
		e-mail: \texttt{silvia.ruda@unimi.it} The author is a member of INDAM-GNAMPA.}
}

\begin{document}
	
	\maketitle
	
	\begin{abstract}
		We consider an  optimal control problem with ergodic (long term average)  reward for
		a  McKean-Vlasov dynamics, where the coefficients of a controlled stochastic differential equation depend on the marginal law of the solution.
		Starting from the associated  infinite time horizon expected discounted reward, we construct both the value $\lambda$ of the ergodic problem and an associated function $\phi$, which provide a viscosity solution to an ergodic Hamilton-Jacobi-Bellman (HJB) equation of elliptic type. In contrast to previous results, we consider the function $\phi$ and the HJB equation on the Wasserstein space, using concepts of derivatives with respect to probability measures.  The pair $(\lambda,\phi)$ also provides information on limit behavior of related optimization problems, for instance, results of Abelian-Tauberian type or limits of value functions of control problems for finite time horizon when the latter tends to infinity.
		Many arguments are simplified by the use of a functional relation for $\phi$ in the form of a suitable dynamic programming principle. 
	\end{abstract}
	
\vspace{5mm}


\vspace{5mm}
	
	\date{}
	
	\section{Introduction}
	This paper is devoted to the study of several classes of optimal control problems for stochastic systems described by McKean-Vlasov stochastic differential equations. 
	They are equations of the form
	\begin{equation}
		\label{introstateequationfinito}
		\begin{cases}
			dX_s= b(X_s, \P_{X_s}, \alpha_s)\,ds + \sigma(X_s, \P_{X_s}, \alpha_s)\,dB_s , \qquad s\in[t,T]\subset [0,T],\\
			X_t= \xi
		\end{cases}
	\end{equation}
	in the case of a finite time horizon $T>0$, or
	\begin{equation}
		\label{introstateequation}
		\begin{cases}
			dX_t= b(X_t, \P_{X_t}, \alpha_t)\,dt + \sigma(X_t, \P_{X_t}, \alpha_t)\,dB_t , \qquad t \ge0,\\
			X_0= \xi,
		\end{cases}
	\end{equation}
	in the case of infinite
	time horizon.
	In these equations, $B$ is a Brownian motion in $\R^d$, $\xi$ is a random initial condition, and $\alpha$ is a control process, taking values in a space $A$ and chosen in a suitable class $\cala$. The special feature of the equations above is the dependence of the coefficients on the marginal law of the solution itself, so that they are functions of $x,\mu,a$, where $x\in\R^d$, $a\in A$ and $\mu$ ranges in $\PPtwo$, the Wasserstein space of Borel probabilities in $\R^d$ with  finite second moment, endowed with the 2-Wasserstein distance. 
	On the coefficients $b,\sigma$ we first impose standard Lipschitz and linear growth conditions, to ensure well-posedness of the state equation. 
	This type of equation was first introduced  in connection with problems in statistical physics. In the controlled case, they arise as limit models for controlled multiagent systems when the number of agents tends to infinity. Their treatment is closely connected with mean-field games, and the two topics share similar mathematical features, see \cite{CarmonaDelarue1}, \cite{CarmonaDelarue2}
	for a comprehensive introduction to the subject.
	
	Optimal control problems for McKean-Vlasov systems have been studied by many authors in the finite horizon case, namely for a reward functional of the form
	\begin{align*}
		J^T(\alpha,\xi,t)=\E\left[
		\int_t^Tf(X_s^{t,\xi,\alpha},\P_{X_s^{t,\xi,\alpha}},\alpha_s)\,ds + g(X_T^{t,\xi,\alpha}, \P_{X_T^{t,\xi,\alpha}})
		\right],
	\end{align*}
	where $X^{t,\xi,\alpha}$ denotes the solution to \eqref{introstateequationfinito}  and $f$ and $g$ both satisfy some suitable growth conditions. 
	
	A possible tool to deal with this problem is the stochastic maximum principle. We will not follow this approach, but we will rather be interested in results based on dynamic programming and the associated Hamilton-Jacobi-Bellman (HJB)  equation satisfied by the value function of the problem. HJB equations can be formulated in two distinct ways: either as  partial differential equations on a Hilbert space of random variables or as  PDEs on the Wasserstein space. In the first case, one defines the value function of the optimal control problem  setting
	\begin{align}
		\label{introlifted}
		V^T(t,\xi)=\sup_{\alpha\in\cala} J^T(\alpha,\xi,t);
	\end{align}
	namely, \(V^T\) is a function of $t\in[0,T]$ and  $\xi \in L^2(\Omega, \calf,\P)$ for a suitable underlying probability space \((\Omega, \calf,\P)\).  
	Using the Hilbert structure of the space $L^2$, as well as the available theory of viscosity solutions to PDEs on Hilbert spaces 
	(see, e.g., \cite{FabbriGozziSwiech}), in \cite{PhamWei1}, \cite{PhamWei2} the HJB equation has been studied as a parabolic PDE on $[0,T]\times L^2$ .
	However, the value function satisfies the so-called law-invariance property, presented in \cite{CossoExistence} under suitable general assumptions: whenever $\xi$ and $\xi'$ are initial conditions with the same law, then one has $V^T(t,\xi)=V^T(t,\xi')$. As a consequence, one can define a value function $v^T$ on $[0,T]\times \PPtwo$ setting
	\begin{equation}
		\label{DefinitionvT}
		v^T(t,\mu)= V^T(t,\xi), \qquad \text{when}\qquad \P_\xi=\mu.
	\end{equation}
	In contrast, the function defined in \eqref{introlifted} will be called the lifted value function.
	It turns out that $v^T$ is a solution to an HJB equation on $[0,T]\times \PPtwo$ of the form
	\begin{equation}
		\label{introhjbuno}
		\begin{cases}
			-\partial_t v^T(t,\mu)- F\left(\mu, \partial_\mu v^T(t,\mu),\partial_x\partial_\mu v^T(t,\mu)\right)=0,\qquad t\in [0,T),\,\mu\in\PPtwo,
			\\
			v^T(T,\mu)= \int_{\R^d}  g(x,\mu)\,\mu(dx),
		\end{cases}
	\end{equation}
	where, for suitable $\varphi:[0,T]\times  \PPtwo\to\R$, we set
	\begin{align*}&
		F\left(\mu, \partial_\mu \varphi( \mu),\partial_x\partial_\mu \varphi( \mu)\right)
		\\&\quad =		\int_{\R^d} \bigg[ \sup \limits_{a \in A} \left\{ f(x, \mu, a ) + \langle b(x, \mu,a ), \partial_{\mu} \varphi(\mu)(x)\rangle    
		+ \frac{1}{2} tr \left(\sigma \sigma^T (x, \mu,a ) \partial_{x}\partial_{\mu}\varphi(\mu)(x) \right) \right\} \bigg]\,\mu(dx).
	\end{align*}
	In \eqref{introhjbuno} the symbol $\partial_\mu$ denotes the so-called $L$-derivative in the sense of Lions; see \cite{CarmonaDelarue1}, or Section \ref{sec-eqsuwasserstein} for some reminders.
	This equation has been studied by many authors and is presently the object of thorough investigation; see, e.g., \cite{CossoExistence}, \cite{CossoUniqueness}, 
	\cite{bayraktaretal2025comparison}, \cite{bayraktaretal2025}.

	One of the aims of our paper is to study the ergodic optimal control problem, where the reward  functional to be maximized is 
	\begin{align}
		J_\infty(\xi,\alpha)=
		\limsup_{T \to\infty} 
		\E \left[\frac{1}{T} \int_{0}^{ T}  f(X_t^{\xi,\alpha}, \P_{X_t^{\xi,\alpha}}, \alpha_t) \,dt \right],
		\label{definitionJinfty}
	\end{align}
	and $X^{\xi,\alpha} $ denotes the solution to the state equation \eqref{introstateequation}. When addressing ergodic problems on $\R^d$ it is natural to impose further  assumptions, namely 
	some dissipativity assumptions on $b$ and $\sigma$ (see below). 
	We also assume boundedness of $f$, which helps in simplifying the proofs.
	Under these assumptions, we will prove that  the value $\sup_{\alpha\in\cala} J_\infty(\xi,\alpha)$ of the problem
	is equal to some constant $\lambda$, i.e., it does not depend on $\xi$. Classical ergodic control problems of stochastic dynamics have long been the subject of study; we will only mention \cite{Bensoussan87} and \cite{ArisawaLions98} as early contributions, 
	and we refer the reader to the book \cite{Araposthatisetal_book} for further details. In recent years, methods based on backward stochastic differential equations have also been used  (see \cite{FuhrmanHuTessitore09},
	\cite{DebusscheHuTessitore11},
	\cite{CossoFuhrmanPham16}, 
	\cite{CossoGuatteriTessitore19}),
	but we will not pursue this methodology. In the aforementioned references, the coefficients of the dynamics are usually required to satisfy suitable dissipativity or periodicity conditions, guaranteeing that the value is in fact  constant. Notice that the long time behavior of reward/cost functionals in averaged form has also been treated in nonergodic cases, where the value varies with the initial condition; for instance, this class of problems is the subject of \cite{BuckdahnGoreacQuincampoix}, where a nonexpansivity condition is adopted in place of the stronger dissipativity condition; see also \cite{BuckdahnLiQuincampoixRenault} for more references on this topic.

	In the classical PDE approach to stochastic ergodic control problems the identification of  the value $\lambda$ is usually coupled with the construction of a suitable function, denoted by $\phi$ in what follows, which is expected to be a solution to an HJB
	equation of elliptic type. The pair $(\lambda,\phi)$ can be used to characterize various asymptotic properties related to optimal control problems.
	
	In the context of McKean-Vlasov systems, 
	this program was first carried out in \cite{BaoTang}. 
	In that paper the authors construct the pair $(\lambda,\phi)$, where the function  $\phi(\xi)$ is mostly considered as a real function defined for $\xi\in L^2$, and they use the theory of viscosity solution on a Hilbert space in order to obtain well-posedness results and investigate the ergodic and other related control problems. 
	It is our purpose to complement these results choosing the alternative approach and constructing the pair 
	$(\lambda,\phi)$, where $\phi$ is treated as a function $\PPtwo\to\R$. Correspondingly, we will eventually prove that $\phi$ is a solution in the viscosity sense to an elliptic PDE on the Wasserstein space.  Indeed, it is not yet clear if well-posedness results for viscosity solutions to a second-order PDE on the space \(\PPtwo\) are immediately implied by analogous results for the lifted version of the same PDE on a suitable \(L^2\) space (as pointed out in Remark 3.6 in \cite{CossoUniqueness}, as well as in Remark 3.2 in \cite{BaoTang}).  A more detailed comparison with the results in \cite{BaoTang} is given below.
	We also mention that some ergodic control problems with mean-field effects have been considered in continuous time with specific forms of the state equation and reward functions (see, e.g., \cite{CannerozziFerrari} and the references therein) and in the context of time-discrete systems (see, for instance, \cite{BayraktarKara} and the references therein).

	We now briefly describe our results and the plan of the paper. 
	After some preliminaries in Section 
	\ref{Sec-prelim}, in Section \ref{SEC-ergodicity} we present a suitable    dissipativity assumption, already introduced in 
	\cite{BaoTang}: we require 
	that   there exist a constant $\eta>0$ such that
	\begin{align}\label{bsigmaergodicintro}
		&
		\E\,\left[ 
		\langle b(\xi, \P_{\xi}, a)- b(\xi', \P_{\xi'}, a),\xi-\xi'\rangle + \frac{1}{2}\, |\sigma(\xi,\P_{\xi},a) - \sigma(\xi',\P_{\xi'},a)|^2\right]\le  -\eta  \E\,\left[ |\xi-\xi'|^2 \right],
	\end{align}
	for every $a \in A$ and every   pair of square summable $\R^d$-valued random variables $\xi,\xi'$; see  Assumptions \ref{HpDissipativity}. Since the direct verification of \eqref{bsigmaergodicintro} may be difficult, in Proposition \ref{BaoTangsimpl}
	we show that it is a consequence of the following requirement: the inequality 
	\begin{align}\label{introbsigmaergodic}
		&\langle b(x, \mu, a)- b(x', \mu, a),x-x'\rangle + \frac{1}{2}\, |\sigma(x,\mu,a) - \sigma(x',\mu,a)|^2\le  -\gamma   |x-x'|^2,
	\end{align}
	holds for all $x,x' \in \R^d$,  $\mu   \in \PPtwo$, $a \in A$, where the constant $\gamma>0$ has to be sufficiently large, depending on the Lipschitz constants of the coefficients. These dissipativity conditions imply suitable ergodicity properties for the solutions to \eqref{introstateequation};
	see Proposition \ref{propdissip}.

	In Section \ref{Sec-varie} 
	we construct the pair $(\lambda,\phi)$ using the vanishing discount method: for any $\beta>0$ we first consider the infinite horizon optimal control problem 
	\begin{equation}
		J_\beta (\xi, \alpha)= \E \left[ \int_{0}^{ \infty} e^{-\beta t} f(X_t^{\xi,\alpha}, \P_{X_t^{\xi,\alpha}}, \alpha_t) \,dt \right].
	\end{equation}
	It was proved in Theorem 4.3 in \cite{Rudà} that the corresponding value only depends on $\xi$ via its law, so that one can in fact define a value function 
	$v^\beta:\PPtwo\to \R$ setting
	\begin{equation*}
		v^\beta(\mu)=	\sup_{\alpha \in \cala}J_\beta (\xi, \alpha), \qquad \text{where}\qquad \P_\xi=\mu.
	\end{equation*}
	By compactness arguments one can prove that for a sequence $\beta_n\to 0$ the limits 
	\[
	\lambda=
	\lim_{n\to\infty} \beta_n \, v^{\beta_n}(\delta_0), \qquad
	\phi(\mu)=\lim_{n\to\infty} (v^{\beta_n}(\mu)- v^{\beta_n}(\delta_0)), 
	\]
	exist, and the function $\phi:\PPtwo\to\R$ is Lipschitz. 
	After this construction we prove the   basic Propositions \ref{phipuntofissodidppprop} and 
	\ref{phiricorsivo}, which will be commented on below after  \eqref{introphipuntofisso}. At the end of this Section 
	we present two results that show the role of the pair $(\lambda,\phi)$ in connection with various optimization problems. We first prove that $(\lambda,\phi)$  can be used to characterize the long time behavior, as the time horizon $T$ tends to infinity,  of the value function $v^T$ introduced in \eqref{DefinitionvT}, namely
	\begin{align}
		\label{v^Tasint}
		|v^T(0,\mu)-\phi(\mu)-\lambda T| \le C\, \left(1+W_2(\mu,\delta_0)^2\right), \qquad 
		\mu\in \PPtwo,
	\end{align}
	where $W_2$ denotes the $2$-Wasserstein distance,
	which implies 
	\begin{align}
		\label{v^Tasintdue}
		\lim_{T\to\infty} \frac{v^T(0,\mu)}{T}=
		\lambda, \qquad 
		\mu\in \PPtwo.
	\end{align}
	Another result concerns the study of the limits
	\begin{align}
		\label{introcomperglim}
		\lim_{\beta \to0} \sup_{\alpha \in \cala} 
		\E \left[\beta \int_{0}^{ \infty} e^{-\beta t} f(X_t^{\xi,\alpha} , \P_{X_t^{\xi,\alpha} }, \alpha_t) \,dt \right],
		\qquad
		\lim_{T \to\infty} \sup_{\alpha \in \cala} 
		\E \left[\frac{1}{T} \int_{0}^{ T}  f(X_t^{\xi,\alpha} , \P_{X_t^{\xi,\alpha} }, \alpha_t) \,dt \right].
	\end{align}
	where $X^{\xi,\alpha} $ is again the solution to   \eqref{introstateequation}.
	Note that both functionals can be written in the form 
	\[
	\E \left[ \int_{0}^{ \infty} k(t) f(X_t^{  \xi, \alpha}, \P_{X_t^{ \xi, \alpha}}, \alpha_t) \,dt \right]
	\]
	where $k$ is the probability density $k(t)= \beta e^{-\beta t}$ or $ k(t)= T^{-1}1_{[0,T]}(t)$,  respectively, and in both cases $k(t)$ tends to zero uniformly on  $[0,\infty)$. Our main result states that both limits in \eqref{introcomperglim} exist and that they are both equal to the constant $\lambda$ (the left limit is close to the definition of $\lambda$, but it is not computed along a subsequence $\beta_n$). In particular,   they do not depend on the initial condition $\xi$. This kind of result is sometimes called a (nonlinear) Abelian-Tauberian theorem.
	
	In Section \ref{sec-eqsuwasserstein}, after some reminders on $L$-derivatives and viscosity solutions on the Wasserstein space, we prove the result, announced above, stating that $(\lambda,\phi)$ is a viscosity solution to the so-called ergodic HJB equation on the Wasserstein space:
	\begin{equation}
		\label{introhjbdue}
		\lambda - F\left(\mu, \partial_\mu \phi(\mu),\partial_x\partial_\mu \phi(\mu)\right)=0,\qquad \mu\in\PPtwo.
	\end{equation}
	Moreover, the value of $\lambda$ is uniquely determined; see Theorem \ref{ergodicHJBuniq}. 
	We also show that the properties \eqref{v^Tasint}-\eqref{v^Tasintdue} can be recast into a result on asymptotic behavior of viscosity solutions to a parabolic PDE of HJB type written forward in time. We emphasize that adding an arbitrary  constant to the function $\phi$ yields another solution to   \eqref{introhjbdue}, but even in classical cases it is known that additional assumptions are required in order  to obtain uniqueness up to an additive constant; 
	see the discussion at the end of Section \ref{sec-eqsuwasserstein}
	and also Remark  \ref{unicitaphi} for more details.
	
	Finally, in Section  \ref{Sec-ergodicontrol}
	we turn to the ergodic control problem with reward functional \eqref{definitionJinfty} and show that $\lambda$ coincides with its value. In order to prove this result we do not need to assume that an optimal control exists. Existence of an optimal control is derived under additional assumptions in the form of a verification theorem applied to the ergodic HJB equation.
	
	We wish to mention that  some papers considering stochastic control of McKean-Vlasov systems on finite horizon deal with a more general case,   also including dependence on the marginal laws of the control process in the coefficients of the controlled equation and the reward functional; for instance, 
	$b(X_s,\P_{X_s}  ,\alpha_s )$ becomes 
	$b(X_s,\P_{X_s}  ,\alpha_s,\P_{\alpha_s} )$. 
	We believe that many of our results extend to this more general case, under natural assumptions for the extended coefficients. However, since for our construction of the basic pair $(\lambda,\phi)$ we rely on the   vanishing discount method, we need corresponding extended results for the optimal control problem on infinite horizon. However, we could not find any such references in the literature. Developing such preliminary results would take a long time and may constitute the object of a separate research project.  For this reason we prefer to concentrate on a simpler basic case where all the essential arguments are presented.
	
	In comparison to \cite{BaoTang}, we should mention that our results are mainly complementary, since they refer to value functions and  HJB equations  defined on the Wasserstein space rather than on the Hilbert space $L^2$. As a consequence, several technical aspects are different. First of all, while choosing the same formulation of the dissipativity condition \eqref{bsigmaergodicintro}, we introduce the stronger condition \eqref{introbsigmaergodic} which helps in verifying the hypothesis directly.  Furthermore, we avoid some boundedness and regularity assumptions on $b,\sigma$ and consider open-loop controls rather than Lipschitz controls in feedback form (except for our verification theorem). 
	
	Finally we would like to point out that many of our results are based on the noteworthy equality
	\begin{align} \label{introphipuntofisso}
		\phi(\mu)+\lambda T= \sup_{\alpha\in\cala} \left\{ \E\left[\int_0^T f(X_t^{\xi,\alpha} , \P_{X_t^{\xi,\alpha} },\alpha_t)\,dt\right] + \phi \Big(\P_{X_T^{\xi,\alpha} }\Big)\right\},
	\end{align} 
	where $\mu=\P_\xi$; see Proposition \ref{phipuntofissodidppprop}. Up to the term $\lambda T$, this identifies the function $\phi$ with the value function of an optimal control problem where the terminal reward is $\phi$ itself.  The use of this formula, or some related ones, simplifies many arguments and, in particular, it allows us to study the ergodic HJB equation by a direct application of known results on HJB for the finite horizon case; see Theorem \ref{ergodicHJB}. In particular, we  avoid the usual deduction of  existence of a solution to the ergodic HJB equation based on stability properties of viscosity solutions to the elliptic HJB, which would require a detailed study in the case of equations on the Wasserstein space.
	We derive \eqref{introphipuntofisso}    starting from the dynamic programming principle for the discounted infinite horizon optimal control problem (Corollary 4.3.1 in \cite{Rudà}) by an argument which is perhaps new.
	
	\section{Notation and preliminaries}
	\label{Sec-prelim}
	
	In the Euclidean space $\R^d$ the standard norm and scalar product are denoted by $|\cdot|$, $\langle\cdot,\cdot\rangle$. For a matrix $\sigma\in\R^{d\times m}$ we set $|\sigma| =(tr(\sigma\sigma^T))^{1/2}$,  where $tr$ denotes the trace and $\sigma^T$ the transpose matrix. 
	For $q\ge1$ we denote by $\calp_q(\R^d)$ the space of Borel probability measures on $\R^d$ with finite moment of order $q$ equipped with the Wasserstein metric $W_q$ (see, for instance,  \cite{CarmonaDelarue1} or \cite{Villani}).

	Let \((\Omega, \calf, \P ) \) be a complete probability space, where an \(m\)-dimensional Brownian motion \(B=(B_t)_{t \geq 0}\) is defined. 
	We assume that there exists a random variable \(U\) independent of $B$ with uniform distribution on \((0,1)\). We denote by $\calg$ the $\sigma$-algebra generated by $U$.
	We will require initial conditions of controlled equations to be \(\calg\)-measurable. Notice that the previous assumptions imply that for every Borel probability  \(\mu\) on  \(\R^d\) there exists a \(\calg \)-measurable   random variable \(\xi\) such that 
	\( \mu= \P_{\xi} \), where the latter symbol denotes the law of $\xi$;    a  proof of this   classical result can be found in   \cite[Lemma 2.1]{CossoExistence}. 
	
	We denote
	by  \(  (\calf_t^B)_{t \geq 0}\) the \(\P \)-completion of the filtration generated by \(B\) and set    \( \F=(\calf_t)_{t \geq 0} \), where \( \calf_t = \sigma( \calg \cup  \calf_t^B). \)
	Let \(A\) be a Polish space, endowed with its Borel $\sigma$-algebra, which we identify with the space of control actions. We define
	the set $\cala$ of admissible controls   as the set of $A$-valued $\F$-predictable
	processes $(\alpha_t)_{t\ge0}$.
	
	\subsection*{Dynamics of the system}
	From now on, we will assume that the coefficients $b,\sigma$ of the state equation satisfy the following assumptions.
	\begin{Hp}
		\label{HpCoefficients}
		\begin{itemize}
			\item[(i)]  The functions
			\[
			b:\R^d\times \PPtwo\times A\to \R^d,
			\qquad
			\sigma:\R^d\times \PPtwo\times A\to \R^{d\times m},
			\]
			are Borel measurable and 
			there exist positive constants $L_{b x}$, $L_{b \mu}$, $L_{\sigma x}$, $L_{\sigma \mu}$, $M$
			such that
			\begin{align}\label{bsigmalip}
				& |b(x, \mu, a)- b(x', \mu, a)| \leq L_{bx}\, |x-x'|, \quad
				|b(x, \mu, a)- b(x, \mu', a)| \leq L_{b\mu}\,W_2(\mu, \mu'), \nonumber
				\\
				&|\sigma(x,\mu,a) - \sigma(x',\mu,a)| 
				\leq L_{\sigma x}   |x-x'| , \quad
				|\sigma(x,\mu,a) - \sigma(x,\mu',a)| 
				\leq L_{\sigma \mu} \,W_2(\mu, \mu'),
				\\
				\label{bsigmagrowth}
				&|b(0, \delta_0, a)|+|\sigma(0, \delta_0, a)|\le M 
			\end{align}
			for all $x,x' \in \R^d$,  $\mu, \mu' \in \PPtwo$, $a \in A$.
			\item[(ii)] For all fixed $(x,\mu) \in \R^d \times \PPtwo$, $b(x,\mu,\cdot)$ and $\sigma(x,\mu,\cdot)$ are continuous.
		\end{itemize}
	\end{Hp}
	
	For $t\in [0,T]$, $\xi \in L^2(\Omega, \calf_t, \P) $, $\alpha\in\cala$ let us consider the state equation
	\begin{equation}
		\label{stateequationfinito}
		\begin{cases}
			dX_s= b(X_s, \P_{X_s}, \alpha_s)\,ds + \sigma(X_s, \P_{X_s}, \alpha_s)\,dB_s , \qquad s\in[t,T],\\
			X_t= \xi.
		\end{cases}
	\end{equation}

	We have the following standard existence and uniqueness result. The case $q=2$ is well known (see, e.g., \cite[Theorem 4.21]{CarmonaDelarue1}, but we could not find a reference for the general case $q\ge 2$, for which we sketch a proof.
	
	\begin{Proposition} \label{wellposedstateeqfinito}
		For $q\ge2$,
		$0\le t\le T$, $\xi \in L^q(\Omega, \calf_t, \P) $, $\alpha\in\cala$, there exists a unique (up to  indistinguishability) continuous adapted solution to \eqref{stateequationfinito} satisfying $\E \,[\sup_{s\in [t,T]}|X_s|^q]<\infty$. Moreover, there exists a constant $C$, which depends only on $T-t,q$, and on the constants in Assumptions \ref{HpCoefficients}, such that
		\[
		\E \,\left[\sup_{s\in [t,T]}|X_s|^q\right]
		\le C     \left(1+\E\,[|\xi|^q]\right).
		\]
	\end{Proposition}

	\begin{proof}
		We introduce the Banach space $\cals^q$ of continuous adapted processes on $[t,T]$
		such that the norm 
		$\|X\|=\left(\E\left[\sup_{s\in [t,T]}|X_s|^q\right]\right)^{1/q}$ is finite.
		Solutions to \eqref{stateequationfinito} coincide with fixed points of the mapping $\Phi: \cals^q\to\cals^q$ defined as
		\begin{equation*}
			\Phi(X)_s= \xi+\int_t^s b(X_r, \P_{X_r}, \alpha_r)\,dr + \int_t^s \sigma(X_r, \P_{X_r}, \alpha_r)\,dB_r , \qquad s\in[t,T].
		\end{equation*}
		We prove that $\Phi$ is a contraction in $\cals^q$, leaving aside the verification that $\Phi$ is well defined. For $X,X'\in\cals^q$, by the Burkholder-Davis-Gundy inequalities, we have
		\begin{align*}
			& \E\left[\sup_{s\in [t,T]} |\Phi(X)_s-\Phi(X')_s|^q\right]
			&\le  
			& \;C \, \E\left[\left(\int_t^T |b(X_s, \P_{X_s}, \alpha_s)-b(X'_s, \P_{X'_s}, \alpha_s)|\,ds \right)^{q}\right]
			\\&&&+ C\, \E\left[\left(
			\int_t^T |\sigma(X_s, \P_{X_s}, \alpha_s) -\sigma(X'_s, \P_{X'_s}, \alpha_s)|^2 \,ds\right)^{ q/2}\right]
			\\&&=:&\;I_1+I_2.
		\end{align*}
		for a constant $C$ depending only on $q$. Letting $\bar X=X-X'$ and using the Lipschitz property of $\sigma$  we have, for some constant $C$ that may differ from line to line,
		\begin{align*}
			I_2& \le C\, \E\left[\left(
			\int_t^T \left( |\bar X_s|^2+W_2\left( \P_{X_s}, \P_{X'_s}\right)^2\right)\,ds\right)^{ q/2}\right]
			\\&\le 
			C\, \E\left[\left(
			\int_t^T   |\bar X_s|^2 \,ds\right)^{ q/2}\right]
			+  C\,  \left(
			\int_t^T \E \left[  |\bar X_s|^2\right] \,ds\right)^{ q/2} 
			\\&\le 
			C\, (T-t)^{q/2} \E \left[\sup_{s\in [t,T]}  |\bar X_s|^q \right]  
			\\&= 
			C\, (T-t)^{q/2} \|\bar X\|^q .
		\end{align*}
		In a similar way one proves that $I_1\le C\, (T-t)^{q}\|\bar X\|^q$, and the required contraction property holds provided that $T-t\le \delta$   for a suitable  $\delta$. The general case follows partitioning the interval $[t,T]$ into subintervals of length $\le\delta$. 
	\end{proof}
	
	We thus denote by $X^{t,\xi,\alpha}=(X_s^{t,\xi,\alpha})_{s\in [t,T]}$ the solution to \eqref{stateequationfinito}. Notice that, with analogous computations, it is also easy to verify that, given two solutions $X=X^{t,\xi,\alpha},X'=X^{t,\xi',\alpha}$ corresponding to the initial conditions $\xi,\xi'\in L^2(\Omega,\calf_t,\P)$, respectively, we have
	\begin{align}
		\label{contdeponxi}
		\E \,\left[\sup_{s\in [t,T]}|X_s-X'_s|^2\right]
		\le C     \,\E\,[|\xi-\xi'|^2],
	\end{align}
	for some constant $C$ that does not depend on $\alpha$.
	
	We will often consider the state equation on the time horizon $[0,\infty)$:
	\begin{equation}
		\label{stateequation}
		\begin{cases}
			dX_t= b(X_t, \P_{X_t}, \alpha_t)\,dt + \sigma(X_t, \P_{X_t}, \alpha_t)\,dB_t , \qquad t \ge0\\
			X_0= \xi,
		\end{cases}
	\end{equation}
	for $q\ge 2$, $\xi\in L^q(\Omega,\calg,\P)$, $\alpha\in\cala$. The solution is well defined on every bounded interval $[0,T]$ and thus for all $t\ge 0$. 
	We will denote by  $X^{\xi,\alpha}$ the solution; notice that, for every fixed $T>0$, this process coincides with the solution $X^{0,\xi,\alpha}$  to the state equation
	\eqref{stateequationfinito} on $[0,T]$.  In what follows, we will thus use both notations.

	\section{Ergodicity}
	\label{SEC-ergodicity}
	
	We start by stating the dissipativity condition that was presented in the introduction.
	
	\begin{Hp} \label{HpDissipativity}
		There exists a constant $\eta>0$ such that
		\begin{align}\label{bsigmaergodic}
			&
			\E\,\left[ 
			\langle b(\xi, \P_{\xi}, a)- b(\xi', \P_{\xi'}, a),\xi-\xi'\rangle + \frac{1}{2}\, |\sigma(\xi,\P_{\xi},a) - \sigma(\xi',\P_{\xi'},a)|^2\right]\le  -\eta  \E\,\left[ |\xi-\xi'|^2 \right],
		\end{align}
		for every $a \in A$ and every   pair of square summable $\R^d$-valued random variables $\xi,\xi'$.
	\end{Hp}
	
	Later in this section we comment on the applicability of \eqref{bsigmaergodic}.
	The main consequence of these assumptions is the following proposition. The statement and proof of 
	\eqref{ergodicxexprimo}
	is the same as  
	Lemma 2.2 in \cite{BaoTang},
	but we report the short proof for convenience. Inequality \eqref{ergodicx} is similar to the conclusion of Lemma 2.1
	in \cite{BaoTang},
	but we dispense with the boundedness conditions on $b,\sigma$ imposed there.

	\begin{Proposition}\label{propdissip}
		Under Assumptions \ref{HpCoefficients} and \ref{HpDissipativity},   for any two solutions $X$, $X'$ to the state equation \eqref{stateequation} we have
		\begin{align}\label{ergodicxexprimo}
			&\E\left[|X_t- X'_t|^2\right]\le \E\left[|X_0- X'_0|^2\right]\,e^{-2\eta t},
			\\\label{ergodicx} &   \E\left[|X_t |^2\right]\le \E\left[|X_0|^2\right]\,e^{-\eta t}+K,
		\end{align}
		for every $t\ge0$, and
		for a constant $K$ that only depends on the constant $\eta$ in \eqref{bsigmaergodic} and on the Lipschitz and growth constants   $ L_{\sigma x},L_{\sigma \mu},M$  in \eqref{bsigmalip} and \eqref{bsigmagrowth}.
	\end{Proposition}
	
	\begin{proof}
		By the It\^o formula, for any $t\ge0$,
		\begin{align*}
			|X_t- X'_t|^2&= |X_0- X'_0|^2 + \int_0^t2 \langle X_s- X'_s, b(X_s, \P_{X_s}, \alpha_s) - b(X'_s, \P_{X'_s}, \alpha_s)\rangle\,ds 
			\\
			&+\int_0^t2 \langle X_s- X'_s, \sigma(X_s, \P_{X_s}, \alpha_s)\,dB_s - \sigma(X'_s, \P_{X'_s}, \alpha_s)\,dB_s\rangle
			\\
			&+ \int_0^t |\sigma(X_s, \P_{X_s}, \alpha_s) - \sigma(X'_s, \P_{X'_s}, \alpha_s)|^2\,ds.
		\end{align*}
		It can be easily verified that the stochastic integrals have zero expectation.
		Taking expectation we thus see that the function $t\mapsto \E[|X_t- X'_t|^2]$ is absolutely continuous, and it follows that
		\begin{align}\label{diffxexprimo}
			\frac{d}{dt}\,\E\, |X_t- X'_t|^2&= 2 \,\E\,\langle X_t- X'_t, b(X_t, \P_{X_t}, \alpha_t) - b(X'_t, \P_{X'_t}, \alpha_t)\rangle 
			\\\nonumber & \quad \,+\E\, |\sigma(X_t, \P_{X_t}, \alpha_t) - \sigma(X'_t, \P_{X'_t}, \alpha_t)|^2\\\nonumber
			&\le   -2\eta\,\E\, |X_t- X'_t|^2,
		\end{align}
		where we use \eqref{bsigmaergodic} in the last inequality. Therefore,  \eqref{ergodicxexprimo} follows immediately.
		
		Arguing as we did for \eqref{diffxexprimo} 
		we obtain
		\begin{align}\label{perxergodico}
			\frac{d}{dt}\,\E\, |X_t |^2= 2 \,\E\,\langle X_t , b(X_t, \P_{X_t}, \alpha_t) \rangle 
			+\E\, |\sigma(X_t, \P_{X_t}, \alpha_t)  |^2. 
		\end{align}
		By the ergodicity condition \eqref{bsigmaergodic},
		\begin{align*} 
			&2\langle b(X_t, \P_{X_t}, \alpha_t),X_t\rangle +  |\sigma(X_t,\P_{X_t},\alpha_t) |^2
			\\
			&=2\langle b(X_t, \P_{X_t}, \alpha_t)- b(0, \delta_0, \alpha_t),X_t\rangle +  2\langle b(0, \delta_0, \alpha_t),X_t\rangle 
			\\&\quad
			+ |\sigma(X_t,\P_{X_t},\alpha_t)-\sigma(0,\delta_0,\alpha_t) +\sigma(0,\delta_0,\alpha_t)|^2
			\\
			& \le  -2\eta    |X_t |^2+ 
			2\langle b(0, \delta_0, \alpha_t),X_t\rangle  
			+ |\sigma(0,\delta_0,\alpha_t) |^2 + 2 \langle
			\sigma(X_t,\P_{X_t},\alpha_t)-\sigma(0,\delta_0,\alpha_t),
			\sigma(0,\delta_0,\alpha_t)\rangle.
		\end{align*}
		From the Lipschitz   conditions 
		\eqref{bsigmalip}  
		we obtain
		\begin{align*} 
			2\langle b(X_t, \P_{X_t}, \alpha_t),X_t\rangle +  |\sigma(X_t,\P_{X_t},\alpha_t) |^2
			&\le   -2\eta   |X_t|^2+ 2 |b(0, \delta_0, \alpha_t)|\, |X_t|+   |\sigma(0,\delta_0,\alpha_t) |^2
			\\&\;\quad+
			2 \,|\sigma(0,\delta_0,\alpha_t) |(L_{\sigma x}\,|X_t|+ L_{\sigma \mu } W_2(\P_{X_t},\delta_0)).
		\end{align*}
		Taking into account the growth condition 
		\eqref{bsigmagrowth}  we have
		$|b(0, \delta_0, \alpha_t)|\le M$ and  $|\sigma(0, \delta_0, \alpha_t)|\le M$. It follows that
		\begin{align*} 
			2\langle b(X_t, \P_{X_t}, \alpha_t),X_t\rangle \!+\!  |\sigma(X_t,\P_{X_t},\alpha_t) |^2 
			\!\le   -2\eta   |X_t|^2+ 
			2M(1+L_{\sigma x}) \,|X_t|+
			2M
			L_{\sigma \mu}
			W_2(\P_{X_t},\delta_0)+M^2
		\end{align*}
		Replacing in \eqref{perxergodico}  and recalling that $W_2(\P_{X_t}, \delta_0) = (\E\, |X_t|^2)^{\frac{1}{2}}$ by definition and 
		$\E\, |X_t|\le (\E\, |X_t|^2)^{1/2}$ by the H\"older inequality, we obtain
		\begin{align}
			\frac{d}{dt}\,\E\, |X_t |^2\le - 2\eta \,\E\, |X_t |^2 + C\,( 1+  (\E\, |X_t |^2)^{1/2}),
		\end{align}
		where $C$ only depends on $M, L_{\sigma x},L_{\sigma \mu}$.
		Using the inequality $ C(\E\, |X_t |^2)^{1/2}\le \eta \E\, |X_t |^2 + C'$, for a suitable constant $C'$, the required conclusion \eqref{ergodicx}
		is easily attained.
	\end{proof}
	
	We conclude this section with some comments on the dissipativity conditions. Since it may be difficult to verify Assumptions \ref{HpDissipativity} directly,  in the following proposition we state a sufficient condition for it to hold.
	
	\begin{Proposition}\label{BaoTangsimpl}
		Suppose that 
		Assumptions \ref{HpCoefficients} hold and  that there exists a constant $\gamma$ such that
		\begin{align}\label{bsigmaergodicbis}
			&\langle b(x, \mu, a)- b(x', \mu, a),x-x'\rangle + \frac{1}{2}\, |\sigma(x,\mu,a) - \sigma(x',\mu,a)|^2\le  -\gamma   |x-x'|^2,
		\end{align}
		for all $x,x' \in \R^d$,  $\mu   \in \PPtwo$, $a \in A$, and
		\begin{align}\label{etaergodicbis}
			\eta:=
			\gamma - \left(L_{b \mu}+L_{\sigma x}\,L_{\sigma \mu}+\frac{1}{2}\,L_{\sigma \mu}^2\right)>0.
		\end{align}
		Then Assumptions \ref{HpDissipativity} hold true.  
	\end{Proposition}

	We postpone the proof and first note that 
	\eqref{bsigmaergodicbis}-\eqref{etaergodicbis}
	may be easily checked in specific cases. For instance, if $b,\sigma$ satisfy Assumptions \ref{HpCoefficients}, then the coefficients
	\[
	b(x, \mu, a)- Kx, \qquad \sigma(x,\mu,a),
	\]
	satisfy 
	\eqref{bsigmaergodicbis}-\eqref{etaergodicbis}, and hence
	Assumptions \ref{HpDissipativity}, provided that $K$ is sufficiently large.
	
	\begin{Remark}
		More generally, when dealing with particular models, it may be possible to verify directly that
		\eqref{ergodicxexprimo} and \eqref{ergodicx} hold true for some constants $K$ and $\eta>0$. In this case the dissipativity condition \eqref{bsigmaergodic} is no longer needed. In other words, in the rest of this paper, one may replace 
		Assumptions \ref{HpDissipativity} by
		the inequalities 
		\eqref{ergodicxexprimo} and \eqref{ergodicx}.
		\qed    
	\end{Remark}
	
	We end this section with the proof of Proposition \ref{BaoTangsimpl}.
	
	\begin{proof}[Proof of Proposition \ref{BaoTangsimpl}]
		From   \eqref{bsigmaergodicbis} it follows that,
		for all $x,x' \in \R^d$,  $\mu,  \mu' \in \PPtwo$, $a \in A$,
		\begin{align*} 
			&2\langle b(x, \mu, a)- b(x', \mu', a),x-x'\rangle +  |\sigma(x,\mu,a) - \sigma(x',\mu',a)|^2
			\\
			&=2\langle b(x, \mu, a)- b(x', \mu, a),x-x'\rangle +  2\langle b(x', \mu, a)- b(x', \mu', a),x-x'\rangle 
			\\& \quad + |\sigma(x,\mu,a)-\sigma(x',\mu,a) +\sigma(x',\mu,a)- \sigma(x',\mu',a)|^2
			\\
			& \le  -2\gamma   |x-x'|^2+ 
			2\langle b(x', \mu, a)- b(x', \mu', a),x-x'\rangle  
			\\&\quad+ |\sigma(x',\mu,a)- \sigma(x',\mu',a)|^2 + 2 \langle
			\sigma(x,\mu,a)-\sigma(x',\mu,a),
			\sigma(x',\mu,a)- \sigma(x',\mu',a)\rangle.
		\end{align*}
		From the Lipschitz conditions \eqref{bsigmalip} we obtain
		\begin{align*} 
			&2\langle b(x, \mu, a)- b(x', \mu', a),x-x'\rangle +  |\sigma(x,\mu,a) - \sigma(x',\mu',a)|^2
			\\
			&\le   -2\gamma   |x-x'|^2+ 2L_{b\mu}\,W_2(\mu, \mu')\, |x-x'|+
			L_{\sigma\mu}^2\,W_2(\mu, \mu')^2 
			+2 L_{\sigma x}L_{\sigma\mu}\, W_2(\mu, \mu')\, |x-x'|.
		\end{align*}
		Then, given $a \in A$ and     $\xi,\xi'\in L^2(\Omega,\calg,\P)$, we deduce that
		\begin{align*} 
			&
			\E\,\left[ 
			\langle b(\xi, \P_{\xi}, a)- b(\xi', \P_{\xi'}, a),\xi-\xi'\rangle + \frac{1}{2}\, |\sigma(\xi,\P_{\xi},a) - \sigma(\xi',\P_{\xi'},a)|^2\right]\\
			&\quad\le  
			-\gamma \,\E\, |\xi- \xi'|^2+
			L_{b\mu}\,W_2( \P_{\xi}, \P_{\xi'})\, \E\, |\xi- \xi'|
			\\&\qquad\,
			+
			\frac{1}{2}\,L_{\sigma\mu}^2\,W_2(\P_{\xi}, \P_{\xi'})^2 
			+ L_{\sigma x}L_{\sigma\mu}\, W_2(\P_{\xi}, \P_{\xi'})\,\E\, |\xi- \xi'|.
		\end{align*}
		Now, as $W_2(\P_{\xi}, \P_{\xi'}) \le (\E\, |\xi- \xi'|^2)^{1/2}$ by definition and 
		$\E\, |\xi- \xi'|\le (\E\, |\xi- \xi'|^2)^{1/2}$ by the H\"older inequality, we have
		\begin{align*} 
			&
			\E\,\left[ 
			\langle b(\xi, \P_{\xi}, a)- b(\xi', \P_{\xi'}, a),\xi-\xi'\rangle + \frac{1}{2}\, |\sigma(\xi,\P_{\xi},a) - \sigma(\xi',\P_{\xi'},a)|^2\right]
			\\&\quad\le  
			( -\gamma +L_{b\mu}+\frac{1}{2}\,
			L_{\sigma\mu}^2+
			L_{\sigma x}L_{\sigma\mu}
			)\,\E\, |\xi- \xi'|^2
			= -\eta\,\E\, |\xi- \xi'|^2,
		\end{align*}
		where $\eta>0$ was defined in \eqref{etaergodicbis}. 
	\end{proof}

	\section{Ergodic problems on the Wasserstein space}
	\label{Sec-varie}
	
	We first introduce assumptions for the running reward function $f$, which will be used in connection with various optimization problems.
	\begin{Hp} \label{HpReward}
		\begin{itemize}
			\item[(i)] The function $f:\R^d\times \PPtwo\times A\to \R$   is Borel measurable and satisfies 
			for some constants $L_f,M_f$,
			\begin{align}\label{iposuf}
				& |f(x,\mu,a)|\le M_f,
				\\  \nonumber
				&|f(x, \mu, a)- f(x', \mu', a)| \leq L_{f}\,\left( |x-x'| +W_2(\mu, \mu')\right), 
			\end{align}
			for every $x,x'\in\R^d$,  $\mu,\mu'\in\PPtwo$, $a\in A$.
			\item[(ii)] For every fixed $(x,\mu) \in \R^d \times \PPtwo$, $f(x,\mu,\cdot)$ is a continuous function.
		\end{itemize}
	\end{Hp}
	In the rest of this Section we assume that Assumptions \ref{HpCoefficients}-\ref{HpReward} hold true. 
	
	In the following subsection we introduce the  pair $(\lambda,\phi)$ described in the introduction, consisting of a real number $\lambda$ and a function $\phi:\PPtwo\to\R$. We then prove the basic Proposition \ref{phipuntofissodidppprop} and its variant Proposition 
	\ref{phiricorsivo}, which will play a crucial role in all of the following.
	In Section \ref{SectionLimitBehavior} we state and prove two results illustrating the relevance of  $(\lambda,\phi)$ to the limit behavior of several optimization problems.
	
	\subsection{The  pair $(\lambda,\phi)$ and its first properties}
	We will use the so-called vanishing discount approach, and to this end we introduce some notation for the corresponding 
	infinite horizon discounted expected reward given by
	\begin{equation}
		J_\beta (\xi, \alpha)= \E \left[ \int_{0}^{ \infty} e^{-\beta t} f(X_t^{  \xi, \alpha}, \P_{X_t^{ \xi, \alpha}}, \alpha_t) \,dt \right]
	\end{equation}
	for    $\xi \in L^2(\Omega, \calg, \P) $,     $\alpha \in \cala$, and
	$\beta>0$, where $X^{  \xi, \alpha}$ is the solution to the state equation \eqref{stateequation}.
	It can be proved (Theorem 4.3 in \cite{Rudà}) that the corresponding value $\sup_{\alpha \in \cala}J_\beta( \xi, \alpha)$  only depends on $\xi$ via its law, so that one can in fact define a value function 
	$v^\beta:\PPtwo\to \R$ setting
	\begin{equation}
		v^\beta(\mu)=	\sup_{\alpha \in \cala}J_\beta (\xi, \alpha), \qquad \mu\in \PPtwo,
	\end{equation}
	for any $\xi \in L^2(\Omega, \calg, \P)$  such that $\P_\xi=\mu$.

	\begin{Proposition} We have, for every $\mu \in\PPtwo$ and $\beta >0$,
		\begin{align}\label{boundvbeta}
			\beta\,|v^\beta(\mu)|\le M_f.
		\end{align}   
		Moreover there exists a constant $L$, depending only  on the constants appearing in Assumptions \ref{HpCoefficients}-\ref{HpReward}, such that
		\begin{align}\label{lipvbeta}
			|v^\beta(\mu)-v^\beta(\nu)|\le L\,W_2(\mu,\nu),
		\end{align}    
		for every $\mu,\nu\in\PPtwo$ and $\beta >0$.
	\end{Proposition}
	
	\begin{proof}
		Since $f$ is bounded by $M_f$ we have $|J_\beta (\xi,\alpha)|\le M_f/\beta$ and 
		\eqref{boundvbeta} follows immediately.
		
		Now let us consider 
		two solutions $X$, $X'$ to the state equation \eqref{stateequation} such that the joint law of $(X_0,X'_0)=(\xi,\xi')$ is an optimal coupling for $W_2$, namely such that $\E[|\xi- \xi'|^2]= W_2(\mu,\nu)^2$ with $\P_{X_0}=\mu$, $\P_{X'_0}=\nu$. The existence of such a pair follows from the existence of the variable $U$ introduced at the beginning of Section
		\ref{Sec-prelim}.
		From \eqref{ergodicxexprimo} we have
		$W_2(\P_{X_t},\P_{X'_t})^2\le     \E[|X_t- X'_t|^2]\le \E[|X_0- X'_0|^2]\,e^{-2\eta t}= W_2(\mu,\nu)^2 \,e^{-2\eta t}$. By the Lipschitz continuity of $f$ (see \eqref{iposuf}) we have,   for every $\alpha\in\cala$,
		\begin{align*}    
			|J_\beta(X_0, \alpha)- J_\beta(X'_0, \alpha)|&\le L_f    \int_{0}^{ \infty} e^{-\beta t} \Big( 
			\E\,|X_t- X'_t|+ W_2(\P_{X_t},\P_{X'_t})\Big) \,dt 
			\\&
			\le  L_f\,  W_2(\mu,\nu)\, \E   \int_{0}^{ \infty}2 e^{-\beta t} e^{-\eta t}    \,dt 
		\end{align*}
		and 
		\eqref{lipvbeta} follows.
	\end{proof}

	Since the family $(\beta\,v^\beta(\delta_0))_{\beta>0}$ is bounded, we can find a subsequence $\beta_n\to 0$ and a $\lambda\in\R$ such that
	\begin{equation}
		\label{deflambda}
		\beta_n\,v^{\beta_n}(\delta_0)\to\lambda,
		\qquad n\to\infty.
	\end{equation}
	Then we have $\beta_n\,v^{\beta_n}(\mu)\to\lambda$ for any other $\mu\in\PPtwo$, since, by 
	\eqref{lipvbeta},
	\[
	|\beta_n\,v^{\beta_n}(\mu)-\lambda|\le
	L\, \beta_n\, W_2(\mu,\delta_0)+
	|\beta_n\,v^{\beta_n}(\delta_0)-\lambda|\to 0, \quad n \rightarrow +\infty.
	\]
	Now consider the sequence   $( v^\beta(\cdot)-v^\beta(\delta_0))_{\beta>0}$. From  \eqref{lipvbeta} it follows that this sequence is equicontinuous and bounded on every $W_2$-ball of $\PPtwo$.
	We also recall that $\PPtwo$ is separable  (see \cite[Theorem 6.18]{Villani}, 
	and \cite{Bolley08} for a short direct proof).	From the Ascoli-Arzelà theorem and a diagonal argument we conclude that there exists a function $\phi:\PPtwo\to\R$ and a subsequence (still denoted $\beta_n$) such that 
	\begin{align}\label{vbetaconvunifcomp}
		v^{\beta_n}(\mu)-v^{\beta_n}(\delta_0)
		\to \phi(\mu), \qquad n\to\infty, 
	\end{align}
	uniformly for $\mu\in K$, where $K$ is any compact subset of $ \PPtwo $. 
	From  \eqref{lipvbeta} we also easily deduce that
	\begin{align}\label{philip}
		|\phi(\mu) |\le L\,W_2(\mu,\delta_0), 
		\qquad     |\phi(\mu)-\phi(\nu)|\le L\,W_2(\mu,\nu),
	\end{align}    
	for every $\mu,\nu\in\PPtwo$.
	
	\begin{Proposition}\label{phipuntofissodidppprop}
		For every $\mu\in\PPtwo$ and $T>0$, for every  $\xi \in L^2(\Omega, \calg, \P) $ such that $\P_\xi=\mu$, we have
		\begin{align}\label{phipuntofissodidpp}
			\phi(\mu)+\lambda T= \sup_{\alpha\in\cala} \left\{ \E\left[\int_0^T f(X_t^{\xi,\alpha}, \P_{X_t^{\xi,\alpha}},\alpha_t)\,dt\right] + \phi \Big(\P_{X_T^{\xi,\alpha}}\Big)\right\}
		\end{align}
	\end{Proposition}
	
	We prepare the proof with a lemma.

	\begin{Lemma}\label{perphipuntofissodidppp}
		If     $\mu= \P_\xi$ belongs to  $\calp_q(\R^d)$ for some $q>2$, then the family of terminal laws 
		$ \left(\P_{X_T^{\xi,\alpha}}\right)_{\alpha\in\cala}$ lies in a compact set of $\PPtwo$.
	\end{Lemma}
	
	\begin{proof}
		By Proposition \ref{wellposedstateeqfinito},
		$\E [|X_T^{\xi,\alpha}|^q]$ is bounded by a constant which does not depend on $\alpha$. It follows that $ \left(\P_{X_T^{\xi,\alpha}}\right)_{\alpha\in\cala}$ lies in a bounded set of  $\calp_q(\R^d)$ and it is therefore relatively compact in  $\PPtwo$ by a known result (see, e.g., Lemma 2.1 in \cite{bayraktaretal2025}).
	\end{proof}
	
	\medskip
	
	\begin{proof}[Proof of Proposition \ref{phipuntofissodidppprop}]
		We start from the dynamic programming principle for the infinite horizon optimal control problem (Corollary 4.3.1 in \cite{Rudà}):
		\begin{equation}
			\label{dppellittico}
			v^\beta(\mu) = 
			\sup_{\alpha\in\cala} \left\{ \E\left[\int_0^T e^{-\beta t}f(X_t^{\xi,\alpha}, \P_{X_t^{\xi,\alpha}},\alpha_t)\,dt\right] + e^{-\beta T}v^\beta \Big(\P_{X_T^{\xi,\alpha}}\Big)\right\},
		\end{equation}
		which gives
		\begin{align*}
			&v^\beta(\mu) -v^\beta(\delta_0) + 
			\frac{1-e^{-\beta T}}{\beta T}\, \beta\,v^\beta(\delta_0)\,T
			\\&\qquad
			= 
			\sup_{\alpha\in\cala} \left\{ \E\left[\int_0^T e^{-\beta t}f(X_t^{\xi,\alpha}, \P_{X_t^{\xi,\alpha}},\alpha_t)\,dt\right] + e^{-\beta T}\left[v^\beta \Big(\P_{X_T^{\xi,\alpha}}\Big)-v^\beta(\delta_0)\right]\right\}.
		\end{align*}
		Now we set $\beta=\beta_n$ and let $n\to\infty$. The left-hand side tends to $\phi(\mu)+\lambda T$. It is clear that
		\[
		\E\left[\int_0^T e^{-\beta t}f(X_t^{\xi,\alpha}, \P_{X_t^{\xi,\alpha}},\alpha_t)\,dt\right]\to
		\E\left[\int_0^T  f(X_t^{\xi,\alpha}, \P_{X_t^{\xi,\alpha}},\alpha_t)\,dt\right]
		\]
		uniformly in $\alpha\in\cala$
		by the boundedness of $f$. We will show that 
		\begin{align}\label{limitepercompatt}
			\sup_{\alpha\in\cala} \left|     v^\beta \Big(\P_{X_T^{\xi,\alpha}}\Big)-v^\beta(\delta_0)- \phi \Big(\P_{X_T^{\xi,\alpha}}\Big)\right|\to 0
		\end{align}
		when $\mu\in \calp_q(\R^d)$ for some $q>2$. Indeed, in this case, by Lemma \ref{perphipuntofissodidppp}, the family $\left(\P_{X_T^{\xi,\alpha}}\right)_{\alpha\in\cala}$ lies in a compact subset $K$ of $\PPtwo$, so the left-hand side of  \eqref{limitepercompatt} is bounded by $\sup_{\mu\in K}  |     v^\beta (\mu)-v^\beta(\delta_0)- \phi (\mu) |$, which tends to $0$ by \eqref{vbetaconvunifcomp}.
		
		So letting $n\to\infty$ it follows that \eqref{phipuntofissodidpp} holds for $\mu\in \calp_q(\R^d)$ for every $q>2$. The general case can be proved by an approximation argument. If $\xi\in L^2(\Omega,\calg,\P)$, we take a sequence $\xi^n\in L^q(\Omega,\calg,\P)$ such that $\E[|\xi^n-\xi|^2]\to 0$, so that $\mu^n=\P_{\xi^n}\to \mu=\P_{\xi}$ in the $W_2$ distance.    We write \eqref{phipuntofissodidpp} for $\xi^n$ and $\mu^n$ and  let $n\to\infty$. On the left-hand side we have $\phi(\mu^n)\to \phi(\mu)$ because of \eqref{philip}. Concerning the right-hand side we have
		\begin{align*}
			&\bigg|  \sup_{\alpha\in\cala} \left\{ \E\left[\int_0^T f(X_t^{\xi^n,\alpha}, \P_{X_t^{\xi^n,\alpha}},\alpha_t)\,dt\right] + \phi \Big(\P_{X_T^{\xi^n,\alpha}}\Big)\right\}
			\\&\qquad\qquad -
			\sup_{\alpha\in\cala} \left\{ \E\left[\int_0^T f(X_t^{\xi,\alpha}, \P_{X_t^{\xi,\alpha}},\alpha_t)\,dt\right] + \phi \Big(\P_{X_T^{\xi,\alpha}}\Big)\right\}
			\bigg|
			\\&\le \sup_{\alpha\in\cala} \E\left[\int_0^T |f(X_t^{\xi^n,\alpha}, \P_{X_t^{\xi^n,\alpha}},\alpha_t)
			-f(X_t^{\xi,\alpha}, \P_{X_t^{\xi,\alpha}},\alpha_t)|
			\,dt\right] 
			+ \sup_{\alpha\in\cala} \left|
			\phi \Big(\P_{X_T^{\xi^n,\alpha}}\Big)-\phi \Big(\P_{X_T^{\xi,\alpha}}\Big)\right|
			\\&\le L_f \, \sup_{\alpha\in\cala} \int_0^T \left[ \E |X_t^{\xi^n,\alpha} - X_t^{\xi,\alpha}|+  W_2\left(\P_{X_t^{\xi^n,\alpha}} 
			, \P_{X_t^{\xi,\alpha}}\right)\right]
			\,dt 
			+ L\, \sup_{\alpha\in\cala} W_2\left(\P_{X_T^{\xi^n,\alpha}} 
			, \P_{X_T^{\xi,\alpha}}\right),
		\end{align*}
		where we used Assumptions \ref{HpReward} and \eqref{philip}    in the last inequality. This can be further estimated by
		\begin{align*}
			L_f \, \sup_{\alpha\in\cala} \int_0^T 2\,\left(\E |X_t^{\xi^n,\alpha} - X_t^{\xi,\alpha}|^2\right)^{1/2} 
			\,dt 
			+ L\, \sup_{\alpha\in\cala} \left(\E |X_T^{\xi^n,\alpha} - X_T^{\xi,\alpha}|^2\right)^{1/2},
		\end{align*}
		which tends to $0$ by \eqref{contdeponxi}. 
		Thus, we can pass to the limit and conclude that \eqref{phipuntofissodidpp} holds for $\xi\in L^2(\Omega,\calg,\P)$ and $\mu=\P_\xi\in \PPtwo$. 
	\end{proof}
	
	In what follows we will need a slight extension of the previous result. Recall that $X^{t,\xi,\alpha}=(X_s^{t,\xi,\alpha})_{  s\in [t,T]}$ denotes the solution to the state equation \eqref{stateequationfinito} starting from \(\xi\) at time \(t\).
	
	\begin{Proposition} \label{phiricorsivo}
		For every $\mu\in\PPtwo$,  $0\le t\le T$, and every  $\xi \in L^2(\Omega, \calf_t, \P) $ such that $\P_\xi=\mu$, we have
		\begin{align*}
			\phi(\mu)+\lambda (T-t)= \sup_{\alpha\in\cala} \left\{ \E\left[\int_t^T f(X_s^{t,\xi,\alpha}, \P_{X_s^{t,\xi,\alpha}},\alpha_s)\,ds\right] + \phi \Big(\P_{X_T^{t,\xi,\alpha}}\Big)\right\}.
		\end{align*}
		
	\end{Proposition}

	\begin{proof}
		Instead of  \eqref{dppellittico}, one starts from the following form of the dynamic programming principle for the infinite horizon optimal control problem: by Theorems 2.9, 3.13 and 4.3 in \cite{Rudà},
		\begin{equation*}
			v^\beta(\mu) = 
			\sup_{\alpha\in\cala} \left\{ \E\left[\int_t^T e^{-\beta (s-t)}f(X_s^{t,\xi,\alpha}, \P_{X_s^{t,\xi,\alpha}},\alpha_s)\,ds\right] + e^{-\beta (T-t)}v^\beta \Big(\P_{X_T^{t,\xi,\alpha}}\Big)\right\}
		\end{equation*}
		and the rest of the proof is completely analogous to the proof of Proposition \ref{phipuntofissodidppprop}  .
	\end{proof}

	\subsection{Applications to limit behavior of some value functions} \label{SectionLimitBehavior}
	We use \eqref{phipuntofissodidpp} to study the asymptotic behavior of the value functions of different optimization problems associated with our McKean-Vlasov dynamics \eqref{stateequationfinito} and \eqref{stateequation}. 
	
	As a first application of the previous constructions and properties, we consider a class of optimal control problems 
	with finite horizon $T>0$ and study the behavior of their value function when $T\to\infty$.   
	We  introduce a measurable terminal reward function $g:\R^d\times \PPtwo\to \R$ satisfying, for some constant $M_g$,
	\begin{align}\label{growthsug}
		|g(x,\mu)|\le M_g\,(1+|x|+W_2(\mu,\delta_0))^2, \qquad
		x\in\R^d,\; \mu\in\PPtwo.
	\end{align}
	We assume that $g$ is uniformly continuous on bounded sets of $\R^d\times \PPtwo$.
	We consider the reward functional
	\begin{equation}\label{finhorizonbis}
		J^T(t,\xi, \alpha)= \E \left[ \int_{t}^{ T}   f(X_s^{t,  \xi, \alpha}, \P_{X_s^{t, \xi, \alpha}}, \alpha_s) \,ds + g\Big(X_T^{t,  \xi, \alpha}, \P_{X_T^{t, \xi, \alpha}}\Big) \right].
	\end{equation}
	It can be proved (Theorem 3.6 in \cite{CossoExistence}) that the corresponding value $\sup_{\alpha \in \cala}J^T(t,\xi, \alpha)$  only depends on $\xi$ via its law, so that one can in fact define a value function $v^T:[0,T]\times \PPtwo\to \R$ setting
	\begin{equation}\label{finhorizon}
		v^T(t,\mu)=	\sup_{\alpha \in \cala}J^T(t,\xi, \alpha), \qquad t\in [0,T],\;\mu\in \PPtwo,
	\end{equation}
	for any $\xi \in L^2(\Omega, \calf_t, \P)$  such that $\P_\xi=\mu$.
	We note that 
	Proposition \ref{phipuntofissodidppprop} shows that $\phi(\mu)+\lambda T$ is the value function at time $t=0$ corresponding to the terminal reward $g=\phi$.
	We are interested in studying the asymptotic behavior of $v^T(0,\mu)$ as $T\to\infty$, which is described in  the following result.
	
	\begin{Theorem}\label{asintoticadiVT}
		We have 
		\[
		|v^T(0,\mu)-\phi(\mu)-\lambda T| \le C\, \left(1+W_2(\mu,\delta_0)^2\right), \qquad 
		\mu\in \PPtwo,
		\]
		where the constant $C$ only depends 
		on the constant $M_g$ in \eqref{growthsug} and on the constants appearing in Assumptions \ref{HpCoefficients}-\ref{HpReward}. In particular,
		\[
		\lim_{T\to\infty} \frac{v^T(0,\mu)}{T}=
		\lambda, \qquad 
		\mu\in \PPtwo.
		\]
	\end{Theorem}
	
	\begin{proof}
		Take   $\xi \in L^2(\Omega, \calg, \P)$  such that $\P_\xi=\mu$.  Comparing 
		\eqref{phipuntofissodidpp} with 
		\begin{equation}\nonumber
			v^T(0,\mu)=	\sup_{\alpha \in \cala}
			\E \left[ \int_{0}^{ T}   f(X_t^{ \xi, \alpha}, \P_{X_t^{  \xi, \alpha}}, \alpha_s) \,ds + g\Big(X_T^{ \xi, \alpha}, \P_{X_T^{ \xi, \alpha}}\Big) \right],
		\end{equation}
		we deduce
		\[
		|v^T(0,\mu)-\phi(\mu)-\lambda T| \le 
		\sup_{\alpha \in \cala}
		\E \left| \phi\Big(\P_{X_T^{ \xi, \alpha}}\Big) - g\Big(X_T^{ \xi, \alpha}, \P_{X_T^{ \xi, \alpha}}\Big) \right|.
		\] 
		By the linear growth property of $\phi$
		(see \eqref{philip}) and by the growth condition \eqref{growthsug} on $g$ we have
		\[
		\E \left| \phi\Big(\P_{X_T^{ \xi, \alpha}}\Big) - g\Big(X_T^{ \xi, \alpha}, \P_{X_T^{ \xi, \alpha}}\Big) \right|\le C \,\left(1+ \E\left[|X_T^{ \xi, \alpha} |^2\right]\right)
		\]
		and the result follows from \eqref{ergodicx}, which gives 
		$\E\left[|X_T^{ \xi, \alpha} |^2\right]\le W_2(\mu,\delta_0)^2 \,e^{-\eta T}+K$.
	\end{proof}
	
	Our next result concerns the limits \eqref{introcomperglim} mentioned in the introduction, and it is in fact an Abelian-Tauberian theorem.

	\begin{Theorem}\label{stessicostiergodici}
		The number $\lambda\in\R$ defined in \eqref{deflambda} satisfies
			\begin{equation}
			\label{deflambdabis}
			\lim_{\beta\to 0}   \beta \,v^{\beta }(\mu)=\lambda,
			\qquad \mu\in\PPtwo.
		\end{equation}
		Moreover,
		\begin{align}
			\label{stessilimerg}
			\lambda=\lim_{\beta \to0} \sup_{\alpha \in \cala} 
			\E \left[\beta \int_{0}^{ \infty} e^{-\beta t} f(X_t^{  \xi, \alpha}, \P_{X_t^{ \xi, \alpha}}, \alpha_t) \,dt \right]=
			\lim_{T \to\infty} \sup_{\alpha \in \cala} 
			\E \left[\frac{1}{T} \int_{0}^{ T}  f(X_t^{  \xi, \alpha}, \P_{X_t^{ \xi, \alpha}}, \alpha_t) \,dt \right],
		\end{align}
		where $X^{  \xi, \alpha}$ is the solution to the state equation \eqref{stateequation} and $\xi \in L^2(\Omega, \calg, \P)$ is any random variable with $\P_{\xi}=\mu$.
	\end{Theorem}
	
	\begin{proof}
		Recall that $ (\beta \,v^{\beta }(\delta_0))_{\beta>0}$ is bounded
		and that 
		$\lambda$ was defined in \eqref{deflambda} as the limit of 
		$ \beta_n\,v^{\beta_n}(\delta_0)$ for a suitable subsequence $\beta_n\to 0$. Now consider the finite horizon control problem \eqref{finhorizon} with terminal reward $g=0$ and note that the right limit in 
		\eqref{stessilimerg} equals 
		$
		\lim_{T \to\infty}T^{-1}v^T(0,\mu)$ and coincides with $\lambda$ by Theorem \ref{asintoticadiVT}.  
		If we had $ \beta_{n'}\,v^{\beta_{n'}}(\delta_0)\to\lambda'\neq\lambda$ for some subsequence $(n')$, then, proceeding as before,  we could extract further subsequences,  construct a limit function $\phi'$,
		and conclude that $
		\lim_{T \to\infty}T^{-1}v^T(0,\mu)=\lambda'$ - a contradiction. It follows  that the limit \eqref{deflambdabis} exists and is given by the first equality in \eqref{stessilimerg}.
	\end{proof}

	\section{Ergodic and parabolic HJB equations on the Wasserstein space}
	\label{sec-eqsuwasserstein}
	
	We still suppose that Assumptions \ref{HpCoefficients}-\ref{HpReward} hold true. In this Section we will show that the pair $(\lambda,\phi)$ constructed in the previous Section is a solution to a fully nonlinear elliptic equation on the space $\PPtwo$, called the ergodic HJB equation. We will also show that $(\lambda,\phi)$ describes the long time behavior of solutions to parabolic HJB equations written forward in time. While these results follow fairly easily from the results presented above, precise statements require a rather lengthy preparation.

	First,	we  need to recall a suitable notion of derivative of real functions on the space \(\PPtwo \).  Following the approach developed, for instance, in  \cite{CossoExistence}, \cite{CossoUniqueness}, \cite{PhamWei1}  for the finite horizon case, we are using the derivative in the lifted sense of Lions (\cite{LionsCours}; see also \cite{Cardaliaguet2013}). 
	
	Given a map \(u:\PPtwo  \rightarrow \R\), we call  \textit{lifting of \(u\)}  any function \(U:L^2(\Omega, \calf , \P ; \R^d) \rightarrow \R\) such that \(U(\xi)=u(\P _{\xi})\) for any \(\xi \in L^2(\Omega, \calf , \P ; \R^d)\).
	We say that   \(u\) is    \textit{L-differentiable} if its lifting \(U\) admits a continuous Fréchet derivative \(D_{\xi}U:L^2(\Omega, \calf , \P ; \R^d) \rightarrow L^2(\Omega, \calf , \P ; \R^d)\).
	In this case, by Theorem 6.5 in  \cite{Cardaliaguet2013},  there exists, for any \(\mu \in \PPtwo \), a measurable function \(\partial_{\mu}u(\mu): \R^d \rightarrow \R^d \) such that \(\P \)-a.s.\ \(D_{\xi}U(\xi)= \partial_{\mu}u(\mu)(\xi) \) for any \(\xi \in L^2(\Omega, \calf , \P ; \R^d)\) with \(\P _{\xi}=\mu\). The function \(\partial_{\mu}u(\mu)\), uniquely defined up to  \(\mu\)-null sets, is called the \textit{Lions derivative of \(u\) at \(\mu\).} 
	
	\begin{Definition}
		We denote by \(\tilde C^{2}(\PPtwo )\) the space of continuous functions \(\varphi: \PPtwo  \rightarrow \R\) such that the following conditions hold: 
		\begin{enumerate}
			\item \(\varphi \) is  L-differentiable;
			\item 
			the map \((\mu, x) \in \PPtwo  \times \R^d \rightarrow \partial_{\mu} \varphi (\mu)(x) \in \R^d\) is jointly continuous;
			\item the derivative \(\partial_x \partial_{\mu} \varphi : (\mu,x) \in \PPtwo  \times \R^d \rightarrow \partial_x \partial_{\mu} \varphi  (\mu)(x) \in \R^{d \times d}\) exists and is jointly continuous;
			\item there exists a constant \(C_{\varphi }\) such that
			\begin{equation}\label{crescitavarphi}
				|\partial_{\mu} \varphi (\mu)(x)| \leq C_\varphi(1+|x|),\quad
				|\partial_x \partial_{\mu} \varphi  (\mu)(x)| \leq C_{\varphi },
				\qquad \mu \in \PPtwo,\,
				x \in \R^d.
			\end{equation}
		\end{enumerate}
	\end{Definition} 
	
	For functions \(\varphi\) also depending on a time parameter,  
	we define the space \(\tilde C^{1,2}([0,T]\times \PPtwo )\) in a similar way.
	\begin{Definition}
		We denote by \(\tilde C^{1,2}([0,T]\times \PPtwo)\) the space of continuous functions \(\varphi: [0,T]\times \PPtwo  \rightarrow \R\) such that the following conditions hold: 
		\begin{enumerate}
			\item \(\varphi \) is  L-differentiable, i.e., its lifting \(\Phi:[0,T] \times L^2(\Omega, \calf, \P) \rightarrow \R\), \(\Phi(t,\xi)=\varphi(t,\P_{\xi})\) for all \(\xi \in L^2(\Omega, \calf, \P)\), admits a continuous Fréchet derivative;
			\item 
			the map \((t, \mu, x) \in [0,T] \times \PPtwo  \times \R^d \rightarrow \partial_{\mu} \varphi (t,\mu)(x) \in \R^d\) is jointly continuous;
			\item the derivatives \(\partial_x \partial_{\mu} \varphi : (t,\mu,x) \in [0,T] \times \PPtwo  \times \R^d \rightarrow \partial_x \partial_{\mu} \varphi  (t,\mu)(x) \in \R^{d \times d}\) and  $\partial_{t} \varphi:  (t,\mu,x) \in [0,T] \times \PPtwo  \times \R^d \rightarrow \partial_{t} \varphi(t,\mu) (x) \in \R$ exist and are jointly continuous;
			\item there exists a constant \(C_{\varphi }\) such that
			\begin{equation}\label{crescitavarphidue}
				|\partial_{\mu} \varphi (t,\mu)(x)| \leq C_\varphi(1+|x|),\quad
				|\partial_x \partial_{\mu} \varphi  (t,\mu)(x)| + |\partial_{t} \varphi(t,\mu) (x)| \leq C_{\varphi }
			\end{equation}
			for all \(\mu \in \PPtwo,\, t \in [0,T], \,
			x \in \R^d.\)
		\end{enumerate}
	\end{Definition} 
	The definition of the space \(\tilde C^{1,2}((0,\infty)\times \PPtwo )\) is completely analogous with \(t \in (0, +\infty)\).
	
	For $\varphi\in \tilde C^2(\PPtwo)$ we set
	\begin{align*}&
		F\left(\mu, \partial_\mu \varphi( \mu),\partial_x\partial_\mu \varphi( \mu)\right)
		\\&\quad :=		\int_{\R^d} \bigg[ \sup \limits_{a \in A} \left\{ f(x, \mu, a ) + \langle b(x, \mu,a ), \partial_{\mu} \varphi(\mu)(x)\rangle    
		+ \frac{1}{2} tr \left(\sigma \sigma^T (x, \mu,a ) \partial_{x}\partial_{\mu}\varphi(\mu)(x) \right) \right\} \bigg]\,\mu(dx).
	\end{align*}
	
	We will be interested both in elliptic and parabolic equations containing the nonlinearity $F$. Given an initial condition $G:\PPtwo\to\R$, we first consider the following equation for an unknown function $w:[0,\infty)\times \PPtwo\to\R$:
	\begin{equation}
		\label{hjbquattro}
		\begin{cases}
			\partial_t w(t,\mu)- F\left(\mu, \partial_\mu w(t,\mu),\partial_x\partial_\mu w(t,\mu)\right)=0,\qquad t> 0,\,\mu\in\PPtwo,
			\\
			w(0,\mu)=    G( \mu),
		\end{cases}
	\end{equation}

	\begin{Definition}
		A continuous function \(w:[0,\infty)\times \PPtwo\to\R\) is called a viscosity subsolution (respectively, supersolution) to \eqref{hjbquattro} if
		\begin{align}
			\label{condizionealbordo}
			w(0,\mu)\le G(\mu), \qquad \mu\in\PPtwo,
		\end{align}
		(respectively, \(w(0,\mu)\ge G(\mu)\) for \(\mu \in \PPtwo\)) and for every \((t,\mu) \in (0,\infty)\times \PPtwo \) and every \(\varphi  \in  \tilde C^{1,2} ((0,\infty)\times  \PPtwo)\) such that \(w- \varphi  \) attains a maximum (respectively, a minimum)  with value \(0\) at \((t,\mu)\), we have
		\begin{align}
			\label{SubsolHJB}
			\partial_t \varphi (t,\mu)- F\left(\mu, \partial_\mu \varphi(t,\mu),\partial_x\partial_\mu \varphi(t,\mu)\right)\le 0
		\end{align}
		(respectively, \(\partial_t \varphi (t,\mu)- F\left(\mu, \partial_\mu \varphi(t,\mu),\partial_x\partial_\mu \varphi(t,\mu)\right)\ge 0\)). A viscosity solution to \eqref{hjbquattro} is a function which is both a supersolution and a subsolution to it.
	\end{Definition}
	
	Below we will also consider parabolic equations which are written backward in time on an interval $[0,T]$ and  with a terminal condition at time $T>0$, as well as elliptic equations for functions depending on $\mu\in\PPtwo$ only. The corresponding definitions of subsolutions and supersolutions require minor changes and should be clear. We underline that the definition is slightly different for an equation of the form
	\begin{equation}
		\label{GeneralErgodicHJB}
		c=F(\mu, \partial_{\mu}w(\mu), \partial_x \partial_{\mu}w(\mu)),
	\end{equation}
	which we will refer to as \textit{ergodic Hamilton-Jacobi-Bellman (HJB) equation}: a viscosity subsolution (respectively, supersolution) to \eqref{GeneralErgodicHJB} is in fact a pair \((c,w)\) given by a constant \(c \in \R\) and a continuous function \(w: \PPtwo \rightarrow \R\) such that, for every \(\mu \in \PPtwo\) and for every \(\varphi  \in  \tilde C^{2} (\PPtwo)\) such that \(w- \varphi  \) attains a maximum (respectively, a minimum) at \(\mu\) with value \(0\), we have
	\begin{align}
		\label{SubsolHJBbis}
		c- F\left(\mu, \partial_\mu \varphi(\mu),\partial_x\partial_\mu \varphi(\mu)\right)\le 0
	\end{align}
	(respectively, \(c- F\left(\mu, \partial_\mu \varphi(\mu),\partial_x\partial_\mu \varphi(\mu)\right)\ge 0\)). A viscosity solution to \eqref{GeneralErgodicHJB} is a pair which is both a supersolution and a subsolution to it.
	
	We now recall the following result, proved in \cite{CossoExistence} (Theorem 5.5), 
	which extends a classical result on viscosity solutions to McKean-Vlasov optimal control problems.
	
	\begin{Theorem}\label{funzionevaloresolviscosa}
		Consider again the finite horizon control problem \eqref{finhorizonbis}-\eqref{finhorizon}
		under the same assumptions  \eqref{growthsug} on $g$. 
		Then the value function $v^T$ defined in \eqref{finhorizon} is   a viscosity solution to the HJB equation
		\begin{equation}
			\label{hjbuno}
			\begin{cases}
				-\partial_t v^T(t,\mu)- F\left(\mu, \partial_\mu v^T(t,\mu),\partial_x\partial_\mu v^T(t,\mu)\right)=0,\qquad t\in [0,T),\,\mu\in\PPtwo,
				\\
				v^T(T,\mu)= \int_{\R^d}  g(x,\mu)\,\mu(dx).
			\end{cases}
		\end{equation}
	\end{Theorem}
	
	We are now ready to state and prove the result on the ergodic HJB equation announced at the beginning of this Section.

	\begin{Theorem} \label{ergodicHJB} The pair $(\lambda,\phi)$ is a viscosity solution  to
		\begin{equation}
			\label{hjbdue}
			\lambda - F\left(\mu, \partial_\mu \phi(\mu),\partial_x\partial_\mu \phi(\mu)\right)=0,\qquad \mu\in\PPtwo.
		\end{equation}
	\end{Theorem}
	
	\begin{proof}
		Proposition \ref{phiricorsivo} states that,  when the terminal reward is chosen as $g(x,\mu)=\phi(\mu)$, the value function is given by $v^T(t,\mu)= \phi(\mu)+\lambda (T-t)$. 
		Writing down \eqref{hjbuno} for this function gives \eqref{hjbdue}; 
		this verification would be trivial for classical solutions, but it is also easy using the definition of viscosity solutions. 
	\end{proof}
	
	We now address for a moment the issue of uniqueness for viscosity solutions to
	\eqref{hjbuno}, which we write in the form
	\begin{equation}
		\label{hjbunobis}
		-\partial_t v(t,\mu)- F\left(\mu, \partial_\mu v(t,\mu),\partial_x\partial_\mu v(t,\mu)\right)=0,
	\end{equation}
	for a generic unknown function $v:[0,T]\times \calp_2(\R^d)\to\R$.
	Uniqueness results are often deduced as a consequence of the following comparison principle:
	
	\medskip
	
	\textbf{(CP)} \begin{em} 
		Let $v_1,v_2:[0,T]\times \calp_2(\R^d)\to\R$ be, respectively, a viscosity subsolution and a viscosity supersolution to \eqref{hjbunobis} 
		on $[0,T)\times \calp_2(\R^d)$, satisfying $v_1(T,\mu)\le v_2(T,\mu)$ for $\mu\in\calp_2(\R^d)$; then  $v_1 \le v_2 $ on $[0,T]\times \calp_2(\R^d)$.    
	\end{em}
	
	\medskip

	Comparison results of this form  are presented in many recent papers (see, e.g., \cite{CossoUniqueness}, \cite{CheungTaiQiu}, \cite{bayraktaretal2025})
	under stronger conditions on the coefficients; for instance, these results usually hold for a diffusion coefficient \(\sigma(x,\mu,a)\) independent of the measure argument \(\mu\) and such that, for all \(a \in A\), \(\sigma(\cdot,  a) \) exhibits a \(C^{2}\)-regularity. Moreover, the functions \(b\), \(\sigma\), \(g\), and \(f\) are often assumed to be bounded and Lipschitz in \((x,\mu)\) uniformly with respect to \(a\), where the Lipschitz continuity with respect to the  measure argument is intended in  the \(W_1\) distance, as the classical Lipschitz continuity with respect to the \(W_2\) distance is not sufficient to guarantee uniqueness (see Remark 2.5 in \cite{CheungTaiQiu}). 
	For our next result 
	we will   assume that \textbf{(CP)} holds for  \eqref{hjbunobis} in the class of  functions which are Lipschitz on $\calp_2(\R^d)$ uniformly in $t$, namely when there exists a constant $C$ such that
	\[
	|v_i(t,\mu)-v_i(t,\mu')|\le C\,W_2(\mu,\mu'), \qquad \mu,\mu'\in\calp_2(\R^d), \,t\in[0,T],\,i=1,2.
	\]
	However, we avoid repeating precise technical conditions in the above references, which would also involve modifications of the concept of viscosity solution, 
	since our results are not related to specific requirements but only to the uniqueness property.
	
	We are now ready for the result on well-posedness and  long time behavior of solutions to  forward parabolic HJB equations announced above.

	\begin{Proposition} \label{longtimehjb} Assume that $g:\R^d\times \PPtwo\to \R$ satisfies \eqref{growthsug} and suppose that \textbf{(CP)} holds for \eqref{hjbunobis} in the class of functions that are Lipschitz on $\calp_2(\R^d)$ uniformly in $t$ on every bounded interval. Then there exists a unique viscosity solution $w:[0,\infty)\times \PPtwo\to\R$ to the following  parabolic partial differential equation of HJB type (written forward in time):
		\begin{equation}
			\label{hjbtre}
			\begin{cases}
				\partial_t w(t,\mu)- F\left(\mu, \partial_\mu w(t,\mu),\partial_x\partial_\mu w(t,\mu)\right)=0,\qquad t\ge 0,\,\mu\in\PPtwo,
				\\
				w(0,\mu)= \int_{\R^d}  g(x,\mu)\,\mu(dx).
			\end{cases}
		\end{equation}

		Moreover we have 
		\[
		|w(t,\mu)-\phi(\mu)-\lambda t| \le C\, \left(1+W_2(\mu,\delta_0)^2\right), \qquad 
		\mu\in \PPtwo,
		\]
		where the constant $C$ only depends 
		on the constant $M_g$ in \eqref{growthsug} and on the constants appearing in Assumptions \ref{HpCoefficients}-\ref{HpReward}. 
		In particular we obtain
		\[
		\lim_{t\to\infty} \frac{w(t,\mu)}{t}=
		\lambda, \qquad 
		\mu\in \PPtwo.
		\]
	\end{Proposition}
	
	\begin{proof}
		We note that,
		on any time interval $[0,T]$, the equality
		\[
		v^T(t,\mu)=w(T-t,\mu), \qquad \mu\in\calp_2(\R^d),\, t\in[0,T],
		\]
		establishes a bijection between viscosity solutions to \eqref{hjbtre} and \eqref{hjbuno} in the class of  functions that are Lipschitz on $\calp_2(\R^d)$ uniformly in $t\in[0,T]$; 
		again, this is immediate for classical solutions but it is readily verified for viscosity solutions as well.  
		
		Since, by Theorem \ref{funzionevaloresolviscosa}, \eqref{hjbuno} has a viscosity solution and uniqueness follows from \textbf{(CP)},
		it is clear that there exists a unique viscosity solution to \eqref{hjbtre}. 
		The other assertions follow from   
		an immediate application of 
		Theorem \ref{asintoticadiVT} (the constant $C$ is the same as in that theorem).
	\end{proof}
	
	We end this Section with a uniqueness result for the ergodic HJB equation \eqref{hjbdue}.

	\begin{Theorem} \label{ergodicHJBuniq} Assume that \textbf{(CP)} holds and suppose that a pair $(\bar\lambda,\bar\phi)$, with $\bar\lambda\in\R$  and $\bar\phi:\calp_2(\R^d)\to\R$ Lipschitz continuous, is a viscosity solution  to \eqref{hjbdue}. Then we have $\lambda=\bar\lambda$.
		
	\end{Theorem}
	
	\begin{proof}
		For any $T>0$ define $\bar v^T(t,\mu):=\bar\phi(\mu)+\bar\lambda(T-t)$. Then $\bar v^T$ is a viscosity solution to the HJB equation
		\begin{equation}
			\begin{cases}
				-\partial_t \bar v^T(t,\mu)- F\left(\mu, \partial_\mu \bar v^T(t,\mu),\partial_x\partial_\mu \bar v^T(t,\mu)\right)=0,\qquad t\in [0,T),\,\mu\in\PPtwo,
				\\
				\bar		v^T(T,\mu)=    \bar \phi(\mu).
			\end{cases}
		\end{equation}
		Since we assume \textbf{(CP)}, we have uniqueness for viscosity solutions to this equation.    Therefore, by 
		Theorem \ref{funzionevaloresolviscosa}, 
		$\bar v^T$ coincides with the value function for the  finite horizon control problem \eqref{finhorizonbis}-\eqref{finhorizon} with terminal condition $g(x,\mu)=\bar\phi(\mu)$, namely 
		\begin{equation}\label{HJBperbarphi}
			\bar v^T(t,\mu)=
			\sup_{\alpha \in \cala}\left\{\E \left[ \int_{t}^{ T}   f(X_s^{t,  \xi, \alpha}, \P_{X_s^{t, \xi, \alpha}}, \alpha_s) \,ds \right]+ \bar\phi\Big(\P_{X_T^{t, \xi, \alpha}}\Big) \right\}, \qquad t\in [0,T],\;\mu\in \PPtwo,
		\end{equation}
		for any $\xi \in L^2(\Omega, \calf_t, \P)$  such that $\P_\xi=\mu$.
		Now recall that by 
		Proposition \ref{phipuntofissodidppprop} we have
		\begin{align*}
			\phi(\mu)+\lambda T= \sup_{\alpha\in\cala} \left\{ \E\left[\int_0^T f(X_t^{\xi,\alpha}, \P_{X_t^{\xi,\alpha}},\alpha_t)\,dt\right] + \phi \Big(\P_{X_T^{\xi,\alpha}}\Big)\right\}.
		\end{align*}
		Setting $t=0$ in \eqref{HJBperbarphi} and subtracting  we deduce
		\begin{align}\label{phiebarphi}
			|\bar\phi(\mu)+\bar \lambda T-\phi(\mu)-\lambda T|=
			|v^T(0,\mu)-\phi(\mu)-\lambda T| \le 
			\sup_{\alpha \in \cala}
			\,\left| \bar\phi\Big(\P_{X_T^{ \xi, \alpha}}\Big)-\phi\Big(\P_{X_T^{ \xi, \alpha}}\Big)  \right|.
		\end{align} 
		By the Lipschitz property of $\phi$ and $\bar\phi$ we have, for some constant $C$,
		\[
		\left| \bar\phi\Big(\P_{X_T^{ \xi, \alpha}}\Big)-\phi\Big(\P_{X_T^{ \xi, \alpha}}\Big)  \right|\le 
		C \,\left(1+ W_2(\delta_0,\P_{X_T^{ \xi, \alpha}})\right)\le
		C \,\left(1+ [\E|X_T^{ \xi, \alpha} |^2]^{1/2}\right),
		\]
		which is bounded by a constant independent of $T$,
		by \eqref{ergodicx}. Dividing \eqref{phiebarphi} by $T$ and letting $T\to\infty$, we obtain $\lambda=\bar \lambda$.
	\end{proof}
	
	This result states that for any pair $(\lambda,\phi)$ solution to the ergodic HJB equation, the number $\lambda$ is uniquely determined. 
	Uniqueness for $\phi$ clearly fails, since $\phi+C$ is also a solution for any constant $C$. Even in simpler situations, other solutions may also exist; see, for instance,
	Remark 1.2 in
	Chapter VII of \cite{BardiCapuzzoDolcetta97}.
	We come back to the uniqueness of $\phi$ (up to a constant) in Remark \ref{unicitaphi} 
	below.

	\section{Ergodic control problem}
	\label{Sec-ergodicontrol}
	
	We still suppose that Assumptions \ref{HpCoefficients}-\ref{HpReward} hold true and denote by 
	$(\lambda,\phi)$ the pair constructed in  Section  \ref{Sec-varie}. 
	In this Section we consider the ergodic reward
	\begin{align} 
		J_\infty(\xi,\alpha)=
		\limsup_{T \to\infty} 
		\E \left[\frac{1}{T} \int_{0}^{ T}  f(X_t^{  \xi, \alpha}, \P_{X_t^{ \xi, \alpha}}, \alpha_t) \,dt \right],
	\end{align}
	where $X^{  \xi, \alpha}$ is the solution to the state equation \eqref{stateequation}, and we maximize 
	$J_\infty(\xi,\alpha)$ over all $\alpha \in \cala$. If the maximum is attained by some $\hat\alpha$, we say that $\hat\alpha$ is optimal for the ergodic problem.
	
	Our first result shows that the value $\sup_{\alpha\in\cala} J_\infty(\xi,\alpha)$
	is constant, i.e., it does not depend on $\xi$, and that it is equal to $\lambda$.
	
	\begin{Proposition} For every $\xi\in L^2(\Omega,\calg,\P)$
		we have 
		\begin{equation}
			\label{lambdaguadagnoergodico}
			\lambda= \sup_{\alpha \in \cala} \;
			\limsup_{T \to\infty} 
			\E \left[\frac{1}{T} \int_{0}^{ T}  f(X_t^{  \xi, \alpha}, \P_{X_t^{ \xi, \alpha}}, \alpha_t) \,dt \right].
		\end{equation}
	\end{Proposition}
	
	\begin{proof}
		We first prove that
		\begin{equation}
			\label{lambdapiugrande}
			\lambda\ge \sup_{\alpha \in \cala} \;
			\limsup_{T \to\infty} 
			\E \left[\frac{1}{T} \int_{0}^{ T}  f(X_t^{  \xi, \alpha}, \P_{X_t^{ \xi, \alpha}}, \alpha_t) \,dt \right].
		\end{equation}
		For any numerical function $q(\alpha,T)$ the inequality 
		\[
		\limsup_{T \to\infty}\; \sup_{\alpha \in \cala} 
		q(\alpha,T)
		\ge\sup_{\alpha \in \cala} \;
		\limsup_{T \to\infty}  q(\alpha,T)
		\]
		is easily verified. Choosing 
		$q(\alpha,T)=\E \left[\frac{1}{T} \int_{0}^{ T}  f(X_t^{  \xi, \alpha}, \P_{X_t^{ \xi, \alpha}}, \alpha_t) \,dt \right]$, 
		the second equality in 
		\eqref{stessilimerg} gives
		$\lambda= \lim_{T \to\infty}\; \sup_{\alpha \in \cala} 
		q(\alpha,T)$ and \eqref{lambdapiugrande} follows.
		
		\bigskip
		
		Next, we prove the opposite inequality. Let us fix $T>0$, $\epsilon>0$, and $\xi \in L^2(\Omega, \calg, \P) $ and set  $\mu=\P_\xi$. 
		By  Proposition  \ref{phiricorsivo},
		for every integer $n\ge0$,   and every  $\zeta \in L^2(\Omega, \calf_{nT}, \P) $  we have
		\begin{align}\label{ricorsione}
			\phi(\P_\zeta)= \sup_{\alpha\in\cala} \left\{ \E\int_{nT}^{(n+1)T} \left[f(X_s^{nT,\zeta,\alpha}, \P_{X_s^{nT,\zeta,\alpha}},\alpha_s)-\lambda \right]\,ds + \phi \Big(\P_{X_{(n+1)T}^{nT,\zeta,\alpha}}\Big)\right\}.
		\end{align}
		Setting $n=0$, $\zeta=\xi$, we see that  there exists $\alpha^0\in\cala$ such that
		\begin{align*}
			\phi(\mu)-\epsilon\, T= \phi(\P_\xi)-\epsilon T\le  \E \int_{0}^{T} \left[f(X_s^{0,\xi,\alpha^0}, \P_{X_s^{0,\xi,\alpha^0}},\alpha^0_s)-\lambda\right]\,ds  + \phi \Big(\P_{X_{ T}^{0,\xi,\alpha^0}}\Big).
		\end{align*}
		Suppose that for some $n\ge 1$ we have constructed a control process $\alpha^{n-1}\in\cala$ such that
		\begin{align}\label{induzionen}
			\phi(\mu)-n\,\epsilon \,T\le  \E \int_{0}^{nT} \left[f(X_s^{0,\xi,\alpha^{n-1}}, \P_{X_s^{0,\xi,\alpha^{n-1}}},\alpha^{n-1}_s)-\lambda\right]\,ds  + \phi \Big(\P_{X_{nT}^{0,\xi,\alpha^{n-1}}}\Big).
		\end{align}
		By \eqref{ricorsione} with $\zeta= X_{nT}^{0,\xi,\alpha^{n-1}}\in L^2(\Omega,\calf_{nT},\P)$ we find $\tilde\alpha^n\in\cala$ such that
		\begin{align}\label{passoinduttivo}
			\phi(\P_\zeta) -\epsilon\, T\le  \E \int_{nT}^{(n+1)T} \left[f(X_s^{nT,\zeta,\tilde\alpha^{n}}, \P_{X_s^{nT,\zeta,\tilde\alpha^n}},\tilde\alpha^n_s)-\lambda\right]\,ds  + \phi \Big(\P_{X_{(n+1)T}^{nT,\zeta,\tilde\alpha^n}}\Big).
		\end{align}
		We define $\alpha^n_s=\alpha^{n-1}_s\,1_{[0,nT) }(s)+ \tilde\alpha^{n}_s\,1_{[nT,\infty) }(s) $ and note that 
		$X_s^{0,\xi,\alpha^n}= X_s^{0,\xi,\alpha^{n-1}}$ for $s\in [0,nT]$
		and
		$X_s^{0,\xi,\alpha^n}=X_s^{nT,\zeta,\tilde\alpha^{n}}$ for $s\ge nT$  by the flow property of solutions to McKean-Vlasov SDEs. Summing
		\eqref{induzionen} and \eqref{passoinduttivo}, we obtain
		\begin{align*}
			\phi(\mu)-(n+1)\,\epsilon \,T\le  \E \int_{0}^{(n+1)T} \left[f(X_s^{0,\xi,\alpha^{n}}, \P_{X_s^{0,\xi,\alpha^{n}}},\alpha^{n}_s)-\lambda\right]\,ds  + \phi \Big(\P_{X_{(n+1)T}^{0,\xi,\alpha^{n}}}\Big)
		\end{align*}
		and conclude by induction that \eqref{induzionen} holds for every $n$. We can also consistently define $\widehat\alpha\in\cala$ setting 
		$\widehat\alpha_s=\alpha^{n-1}_s$ for $s\in [0,nT)$ so that $X_s^{0,\xi,\widehat\alpha}=X_s^{0,\xi,\alpha^{n-1}}$ for $s\in [0,nT)$.
		Rewriting \eqref{induzionen} in terms of $\widehat\alpha $ and $X_s^{0,\xi,\widehat\alpha}$ and 
		dividing by $nT$, we obtain
		\begin{align*} 
			\frac{1}{nT}\,   \phi(\mu)
			-\frac{1}{nT}\,\phi \Big(\P_{X_{nT}^{0,\xi,\widehat\alpha}}\Big) +\lambda - \epsilon  \le \frac{1}{nT} \,\E \int_{0}^{nT}  f(X_s^{0,\xi,\widehat\alpha}, \P_{X_s^{0,\xi,\widehat\alpha}},\widehat\alpha_s) \,ds   .
		\end{align*}
		Since, by \eqref{philip}
		and \eqref{ergodicx},
		\begin{equation}
			\label{passoallimit}
			\left| \phi \Big(\P_{X_{nT}^{0,\xi,\widehat\alpha}}\Big)\right|\le   L\, W_2 \Big(\P_{X_{nT}^{0,\xi,\widehat\alpha}},\delta_0\Big)\le L\,\left(
			\E\left[|X_{nT}^{0,\xi,\widehat\alpha} |^2\right]\right)^{1/2}\le  L\,\left(\E\left[|\xi|^2\right]\,e^{-\eta nT}+K\right)^{1/2},
		\end{equation}
		we obtain
		\begin{align*} 
			\lambda - \epsilon  \le \limsup_{n\to\infty}\,\frac{1}{nT} \,\E \int_{0}^{nT}  f(X_s^{0,\xi,\widehat\alpha}, \P_{X_s^{0,\xi,\widehat\alpha}},\widehat\alpha_s) \,ds  \le \limsup_{T'\to\infty}\,\frac{1}{T'} \,\E \int_{0}^{T'}  f(X_s^{0,\xi,\widehat\alpha}, \P_{X_s^{0,\xi,\widehat\alpha}},\widehat\alpha_s) \,ds   .
		\end{align*}
		It follows that 
		\begin{align*} 
			\lambda - \epsilon  \le  \sup_{\alpha \in \cala} \;
			\limsup_{T \to\infty} \,\frac{1}{T}\,
			\E   \int_{0}^{ T}  f(X_t^{ 0, \xi, \alpha}, \P_{X_t^{ 0,\xi, \alpha}}, \alpha_t) \,dt  
		\end{align*}
		and the required inequality is obtained letting $\epsilon\to 0$.
	\end{proof}
	
	\begin{Remark}
		For the reader's convenience we sketch another alternative proof of the inequality \eqref{lambdapiugrande}, based on an argument which dates back at least to \cite{Arisawa1998} (Theorem 5) in the case of deterministic control; see also \cite[Theorem 3.4]{BaoTang} for a recent application to stochastic control problems with mean-field effects.
		We start again from
		\eqref{dppellittico}, which gives, for every 
		$\alpha\in\cala$ and setting $\mu=\P_\xi$,
		\[
		v^\beta(\mu) \ge 
		\E\left[\int_0^T e^{-\beta t}f(X_t^{\xi,\alpha}, \P_{X_t^{\xi,\alpha}},\alpha_t)\,dt\right] + e^{-\beta T}v^\beta \Big(\P_{X_T^{\xi,\alpha}}\Big) .
		\]
		Multiplying by $\beta$,
		\begin{align*}&
			\beta\,v^\beta(\mu)-\beta e^{-\beta T}v^\beta \Big(\P_{X_T^{\xi,\alpha}}\Big) \ge 
			\E\left[\int_0^T  \beta e^{-\beta t}f(X_t^{\xi,\alpha}, \P_{X_t^{\xi,\alpha}},\alpha_t)\,dt\right] 
			\\&
			\qquad = \E\left[\int_0^T  \beta\, f(X_t^{\xi,\alpha}, \P_{X_t^{\xi,\alpha}},\alpha_t)\,dt\right] + 
			\E\left[\int_0^T  \beta\, ( e^{-\beta t}-1)\,f(X_t^{\xi,\alpha}, \P_{X_t^{\xi,\alpha}},\alpha_t)\,dt\right] .
		\end{align*}
		The last expectation is bounded by $M_f \int_0^T  \beta\, | e^{-\beta t}-1|\, dt= M_f\,(\beta T +e^{-\beta T}-1)$. Setting $\delta=\beta T$,
		\[
		\beta\,v^\beta(\mu)- e^{-\delta}\beta\,v^\beta \Big(\P_{X_T^{\xi,\alpha}}\Big) \ge 
		\delta\,\E\left[\frac{1}{T}\int_0^T    f(X_t^{\xi,\alpha}, \P_{X_t^{\xi,\alpha}},\alpha_t)\,dt\right]  
		- M_f\,(\delta +e^{-\delta}-1).
		\]
		Now we let $\beta\to 0$ (possibly just along a subsequence) and $T\to\infty$, keeping $\delta$ fixed. We recall  that $\beta\,v^\beta(\mu)\to \lambda$. Since, by \eqref{lipvbeta}
		and \eqref{ergodicx},
		$\beta\,\left| v^\beta(\mu)-   v^\beta \Big(\P_{X_T^{\xi,\alpha}}\Big)\right|\le \beta\, L\, W_2 \Big(\P_{X_T^{\xi,\alpha}},\mu\Big)
		\to 0$,
		we conclude that $\beta\,v^\beta \Big(\P_{X_T^{\xi,\alpha}}\Big)\to\lambda$.  It follows that
		\[
		\lambda\,(1- e^{-\delta}) \ge 
		\delta\,\limsup_{T\to\infty}\E\left[\frac{1}{T}\int_0^T    f(X_t^{\xi,\alpha}, \P_{X_t^{\xi,\alpha}},\alpha_t)\,dt\right]  
		- M_f\,(\delta +e^{-\delta}-1).
		\]
		Dividing by $\delta$ and letting $\delta\to 0$, we obtain
		$\lambda  \ge 
		\limsup_{T\to\infty}\E\left[\frac{1}{T}\int_0^T    f(X_t^{\xi,\alpha}, \P_{X_t^{\xi,\alpha}},\alpha_t)\,dt\right]  $ and  \eqref{lambdapiugrande} follows.
		
		One may wonder whether this argument  can be adapted to yield the opposite inequality as well, according to some results which can be found in the literature.  Given $\delta>0$ one can find a control $\widetilde \alpha$ such that  
		\[
		v^\beta(\mu) -\delta^2 \le 
		\E\left[\int_0^T e^{-\beta t}f(X_t^{\xi,\widetilde\alpha}, \P_{X_t^{\xi,\widetilde\alpha}},\widetilde\alpha_t)\,dt\right] + e^{-\beta T}v^\beta \Big(\P_{X_T^{\xi,\widetilde\alpha}}\Big) 
		\]
		and, after similar rearrangements and estimates, one can let $T\to\infty$ and $\beta\to 0$, keeping $\delta=\beta T$ fixed as before. However, to finish the proof by this argument, one needs to find the control $\widetilde \alpha$ independent of $T$ and $\beta$. While dependence on $T$ can be avoided (using a more general version of the dynamic programming principle; see, e.g., 
		\cite[Chapter VII Proposition 1.3]{BardiCapuzzoDolcetta97} for deterministic control), the dependence on $\beta$ seems unavoidable. In order to circumvent this difficulty we have devised the previous proof. 
		\qed
	\end{Remark}
	
	We end this Section with some results on the existence of  optimal controls for the ergodic problem.
	
	\begin{Proposition} \label{lemmaoptimality} Let  $\xi \in L^2(\Omega, \calg, \P) $ and denote  $\mu=\P_\xi$. 
		Suppose that $\hat\alpha\in\cala$ satisfies one of the  following equivalent conditions:
		\begin{enumerate}
			\item [i)] for every $T>0$,
			\begin{align}\label{phipuntofissodidppbis}
				\phi(\mu)+\lambda T=  \E\left[\int_0^T f(X_t^{\xi,\hat\alpha}, \P_{X_t^{\xi,\hat\alpha}},\hat\alpha_t)\,dt\right] + \phi \Big(\P_{X_T^{\xi,\hat\alpha}}\Big);
			\end{align}
			\item[ii)] for every $T>0$, $\hat\alpha$ is optimal for the finite horizon problem
			\begin{align}\label{phipuntofissodidppter}
				\sup_{\alpha\in\cala} \left\{ \E\left[\int_0^T f(X_t^{\xi,\alpha}, \P_{X_t^{\xi,\alpha}},\alpha_t)\,dt\right] + \phi \Big(\P_{X_T^{\xi,\alpha}}\Big)\right\}.
			\end{align}
		\end{enumerate}
		Then $\hat\alpha$ is optimal for the ergodic problem.
	\end{Proposition}
	
	\begin{proof}
		The equivalence of $(i)$ and $(ii)$ follows from  
		Proposition \ref{phipuntofissodidppprop}. If  \eqref{phipuntofissodidppbis} holds, then dividing by $T$ and letting $T\to\infty$,  using again 
		\eqref{philip}
		and \eqref{ergodicx} as before (see, for instance,  
		\eqref{passoallimit}), we obtain
		\begin{align*} 
			\lambda = \lim_{T\to\infty} \frac{1}{T} \E\left[\int_0^T f(X_t^{\xi,\hat\alpha}, \P_{X_t^{\xi,\hat\alpha}},\hat\alpha_t)\,dt\right] 
		\end{align*}
		which proves the optimality of $\hat\alpha$.
	\end{proof}
	
	This proposition leads to a sufficient condition for optimality in the form of a martingale optimality principle,  as proposed in 
	\cite[Proposition 6.1-$(ii)$]{CossoGuatteriTessitore19}.
	\begin{Corollary}
		Suppose that there exists $\hat\alpha\in\cala$ such that
		\begin{align*}
			\phi \Big(\P_{X_T^{\xi,\hat\alpha}}\Big) +     \int_0^T f(X_t^{\xi,\hat\alpha}, \P_{X_t^{\xi,\hat\alpha}},\hat\alpha_t)\,dt  -\lambda\,T, \qquad T\ge0,
		\end{align*}
		is a martingale.  
		Then $\hat\alpha$ is optimal for the ergodic control problem.
	\end{Corollary}
	
	The proof is obvious, as taking expectation we obtain \eqref{phipuntofissodidppbis}. 
	
	\begin{Remark}
		Proposition \ref{lemmaoptimality} admits the following extension. 
		Suppose that there exists a pair $(\lambda',\phi')$ such that the function $\phi':\calp_2(\R^d)\to \R$ satisfies the quadratic growth condition  $|\phi'(\mu)|\le C(1+W_2(\mu,\delta_0)^2)$ for some constant $C$ and every $\mu\in\calp_2(\R^d)$, and such that the analogue of \eqref{phipuntofissodidpp} holds, namely for every 
		$\mu\in\PPtwo$,  $T>0$, and  $\xi \in L^2(\Omega, \calg, \P) $ such that $\P_\xi=\mu$ we have
		\begin{align} \label{FundamentalEquation'}
			\phi'(\mu)+\lambda' T= \sup_{\alpha\in\cala} \left\{ \E\left[\int_0^T f(X_t^{\xi,\alpha}, \P_{X_t^{\xi,\alpha}},\alpha_t)\,dt\right] + \phi'\Big(\P_{X_T^{\xi,\alpha}}\Big)\right\}.
		\end{align}
		If we consider the finite horizon control problem \eqref{finhorizonbis}-\eqref{finhorizon} with arbitrary terminal value $g$ (for instance, $g=0$) and we call $v^T$ the corresponding value function, then 
		arguing as in Theorem \ref{asintoticadiVT} we first deduce that
		\[
		\lambda'= 
		\lim_{T\to\infty} \frac{v^T(0,\mu)}{T}=
		\lambda, \qquad 
		\mu\in \PPtwo.
		\]
		Next also suppose that there exists
		$\hat\alpha\in\cala$ such that 
		for every $T>0$ we have
		\begin{align} \label{OptimalControlCondition}
			\phi'(\mu)+\lambda T=  \E\left[\int_0^T f(X_t^{\xi,\hat\alpha}, \P_{X_t^{\xi,\hat\alpha}},\hat\alpha_t)\,dt\right] + \phi' \Big(\P_{X_T^{\xi,\hat\alpha}}\Big),
		\end{align}
		which is the analogue to \eqref{phipuntofissodidppbis}.  Then    $\hat\alpha$ is optimal for the ergodic problem. 
		The proof remains the same as in Proposition \ref{lemmaoptimality}.
		\qed 
	\end{Remark}
	
	Another possibility to construct an optimal control by checking the conditions in Proposition  \ref{lemmaoptimality}  makes use of the ergodic HJB equation for $\phi$. It can be stated in the form of a verification theorem.
	Recall the definition of the space $\tilde C^2(\PPtwo)$ given in the previous Section.
	
	\begin{Theorem} \label{teoverifica} Suppose that $\phi\in \tilde C^2(\PPtwo)$ and that the pair $(\lambda,\phi)$ is a classical solution  to the ergodic equation \eqref{hjbdue}. Suppose that there exists a Borel measurable function $\underline{\alpha}:\R^d\times\PPtwo\to A$ that realizes the supremum in the equation, namely satisfying
		\begin{align}
			\label{HJBfeedback}
			\lambda  = &\int_{\R^d} \biggl\{ f(x, \mu, \underline{\alpha}(x,\mu) ) + \langle b(x, \mu,\underline{\alpha}(x,\mu) ), \partial_{\mu} \phi(\mu)(x)\rangle \nonumber \\
			&		 + \frac{1}{2} tr \Big(\sigma \sigma^T (x, \mu,\underline{\alpha}(x,\mu) ) \partial_{x}\partial_{\mu}\phi(\mu)(x) \Big) \biggr\}\,\mu(dx),
			\qquad x\in\R^d,\, \mu\in\PPtwo .
		\end{align}
		Finally assume that the closed-loop equation
		\begin{equation}
			\begin{cases}
				dX_t= b\Big(X_t, \P_{X_t}, \underline{\alpha}(X_t,\P_{X_t})\Big)\,dt + \sigma\Big(X_t, \P_{X_t}, \underline{\alpha}(X_t,\P_{X_t})\Big)\,dB_t , \qquad t \ge0\\
				X_0= \xi,
			\end{cases}
		\end{equation}
		admits a solution \(X\) for
		some  $\xi\in L^2(\Omega,\calg,\P)$ and that the control
		\[
		\hat\alpha_t:=\underline{\alpha}(X_t, \P_{X_t}), \qquad t\ge0,
		\]
		belongs to $\cala$. Then $\hat\alpha$ is optimal for the ergodic problem.
	\end{Theorem}
	
	\begin{proof}
		We fix   $T>0$ arbitrary and apply the It\^o formula on the interval $[0,T]$ (see, e.g., Theorem 4.16 and Remark 4.17 in \cite{CossoExistence}) to the composition of \(\varphi\) with the solution of the closed-loop equation, obtaining
		\begin{align*}
			\phi \Big(\P_{X_T^{\xi,\hat\alpha}}\Big) -\phi(\mu)=\; &\E\int_0^T \biggl\{  \langle b(X_t,\P_{X_t},\underline{\alpha}(X_t,\P_{X_t}) ), \partial_{\mu} \phi(\P_{X_t})(X_t)\rangle   \\
			&		 + \frac{1}{2} tr \Big(\sigma \sigma^T (X_t,\P_{X_t},\underline{\alpha}(X_t,\P_{X_t}) ) \partial_{x}\partial_{\mu}\phi(\P_{X_t})(X_t) \Big) \biggr\}\,dt .
		\end{align*}
		From \eqref{HJBfeedback} it follows that
		\begin{align*}
			\phi \Big(\P_{X_T^{\xi,\hat\alpha}}\Big) -\phi(\mu)= \E\int_0^T \biggl\{ \lambda- f(X_t,\P_{X_t}, \underline{\alpha}(X_t,\P_{X_t}) ) \biggr\}\,dt 
			=  \lambda\,T - \E\int_0^T   f(X_t,\P_{X_t},  \hat\alpha_t ) \,dt .
		\end{align*}
		Therefore \eqref{phipuntofissodidppbis}
		holds and the conclusion follows from Proposition 
		\ref{lemmaoptimality}.
	\end{proof}
	
	\begin{Remark}
		\label{unicitaphi}
		We recall that  for a pair $(\lambda,\phi)$ which is a viscosity solution to the ergodic HJB equation \eqref{hjbdue},  uniqueness holds for the number $\lambda$.     In the classical framework of ergodic control, assuming more regularity on $\phi$ and some of  the assumptions of the previous theorem (and suitable additional conditions), one can prove that the function $\phi$ is also uniquely determined, up to a constant. This conclusion is often reached as a consequence of results on long time behavior of the solution $w$ to the parabolic equation \eqref{hjbtre}: setting
		\[
		G(t,\mu)=
		w(t,\mu)-\phi(\mu)-\lambda t,
		\qquad t\ge0,\,\mu\in\calp_2(\R^d),
		\]
		besides the conclusion that
		\(\sup_{t\ge0}|G(t,\mu)|<\infty
		\),
		reached in 
		Proposition  \ref{longtimehjb},   one may conjecture that $G(t,\mu)$ tends to a constant as $t\to\infty$, thus proving the required uniqueness for $\phi$. We refer the reader to, for instance, \cite{Ichihara2012}, \cite{IchiharaSheu2013} and the references therein on the large literature on asymptotic behavior for HJB equations on $\R^d$, as well as \cite{CossoFuhrmanPham16} for a probabilistic and control-theoretic approach to the same issue.  It seems that there is no immediate extension of these results to our framework.     Moreover, the proofs in the classical case also rely on existence and properties of an invariant measure for the process  $X^{\xi,\hat\alpha}$ solution to the closed-loop equation. Since the latter is non-Markovian, because of the occurrence of mean-field effects, even this part of the proof would require a detailed investigation and probably more restrictive assumptions. In particular, existence, uniqueness, and (exponential) convergence to invariant measures for uncontrolled McKean-Vlasov systems have been studied in many contributions; see, e.g., \cite{Malrieu}, \cite{Veretennikov}, and \cite{Butkovsky} for the case of additive noise,  \cite{WangLandau} for time-homogeneous coefficients, and \cite{Wang}, \cite{RenWang} and references therein for distribution-dependent drift terms; the results in the aforementioned references hold under additional assumptions on the coefficients of the state equation (e.g.\, monotonicity and Lyapunov conditions, boundedness, and nondegeneracy) and should be adapted to the case of a stochastic differential equation in closed-loop form.
		For these reasons we will not present results on uniqueness for the function $\phi$ in this paper.
	\end{Remark}

\end{document}